\documentclass[12pt,a4paper]{amsart}

\usepackage[utf8]{inputenc}
\usepackage[T1]{fontenc}
\usepackage[english]{babel}
\usepackage{cite}

\usepackage{indentfirst}
\usepackage{amssymb}
\usepackage{amsfonts}
\usepackage{amsmath}
\usepackage{amsthm}
\usepackage[top=2.5cm, bottom=2.5cm, left=2.5cm, right=2.5cm]{geometry}
\usepackage{amsopn}
\usepackage[colorlinks]{hyperref}
\usepackage{mathrsfs}
\usepackage{amsxtra}
\usepackage{color}
\usepackage{dsfont}
\usepackage{cleveref}

\newcommand{\al}{\alpha}
\newcommand{\be}{\beta}
\newcommand{\si}{\sigma}
\newcommand{\om}{\omega}

\theoremstyle{plain}
\newtheorem{thm}{Theorem}

\newtheorem{lem}[thm]{Lemma}

\newtheorem{cor}[thm]{Corollary}

\theoremstyle{definition}

\newtheorem*{example*}{Example}

\newtheorem*{rem*}{Remark}
\newtheorem{rem}[thm]{Remark}

\newcommand{\dd}{d_{\Omega}}
\newcommand{\meas}{\left|\mathbb{S}^{d-1}\right|}

\newcommand{\Om}{\Omega}
\newcommand{\ga}{\gamma}
\newcommand{\ve}{\varepsilon}
\newcommand{\va}{\varphi}
\newcommand{\Sd}{\mathbb{S}^{d-1}}

\newcommand{\R}{\mathbb{R}}
\newcommand{\Rplus}{\mathbb{R}^d_+}

\DeclareMathOperator{\sgn}{sgn}

\DeclareMathOperator{\B}{B}

\DeclareMathOperator{\2F1}{_2F_1}

\DeclareSymbolFont{bbsymbol}{U}{bbold}{m}{n}
\DeclareMathSymbol{\ind}{\mathbin}{bbsymbol}{'061}

\title[Fractional Hardy--Maz'ya inequality on a half-space]{Fractional Hardy--Maz'ya inequality on a half-space}
\author[M{.} Kijaczko]{Micha\l{} Kijaczko}
\author[A{.} Szczukiewicz]{Antoni Szczukiewicz}


\address[ M.K. and A.S.]{Faculty of Pure and Applied Mathematics\\ Wroc{\l}aw University 
	of Science and Technology\\
	Wybrze\.ze Wyspia\'nskiego 27,
	50-370 Wroc{\l}aw, Poland
}
\email{michal.kijaczko@pwr.edu.pl}
\email{antoniszczukiewicz037@gmail.com, 264032@student.pwr.edu.pl}

\thanks{The second-named author was supported by the OPUS grant 25
2023/49/B/ST1/00678 of the National Science Centre (Poland)}

\begin{document}
\maketitle
\begin{abstract}
The main purpose of this article is to provide a fractional counterpart of the well-known Maz'ya inequality on the half-space, that is
$$
\int_{\Rplus}\int_{\Rplus}\frac{|u(x)-u(y)|^p}{|x-y|^{d+sp}}dy\,dx\ge\mathcal{D}_{d,s,p}\int_{\Rplus}\frac{|u(x)|^p}{x_d^{sp}}dx+C_{d,s,p,\tau}\int_{\Rplus}\frac{|u(x)|^p}{x_{d}^{sp-\tau}\left(x_{d-1}^2+x_d^2\right)^{\tau/2}}dx,
$$
where $\mathcal{D}_{d,s,p}$ stands for the sharp constant in the fractional Hardy inequality on a half-space $\Rplus$. We also obtain a similar result in the setting of Sobolev--Bregman forms.
\end{abstract}
\section{Introduction}
\subsection{Classical Hardy inequalities}
Hardy inequalities play a central role in mathematical analysis, due to their connections with various fields, such as partial differential equations, mathematical physics, probability theory and stochastic processes, to name just a few. Recall that the one-dimensional Hardy inequality is 
$$
\int_{0}^{\infty}|u'(x)|^p\,dx\ge\left(\frac{p-1}{p}\right)^p\int_{0}^{\infty}\frac{|u(x)|^p}{x^p}dx,\quad p>1,u\in C_c^1((0,\infty)).
$$
This inequality can be generalised in many ways. We will be particularly interested in the case of the half-space
$$
\Rplus:=\left\{x=(x_1,\dots,x_d)\in\R^d\colon x_d>0\right\}.
$$
A related Hardy inequality is of the form 
\begin{equation}\label{classicalHardyhalfspace}
\int_{\Rplus}|\nabla u(x)|^p\,dx\ge\left(\frac{p-1}{p}\right)^p\int_{\Rplus}\frac{|u(x)|^p}{x_d^p}dx,\quad u\in C_c^1\left(\Rplus\right).    
\end{equation}
The constant $\left((p-1)/p\right)^p$ is sharp in \eqref{classicalHardyhalfspace}, that is, it cannot be replaced by a bigger one. Moreover, it can be shown that there exists no function $u$, not constantly equal to zero, such that \eqref{classicalHardyhalfspace} becomes an equality. Therefore, it is natural to search for lower bounds for the so-called \emph{Hardy difference}
$$
I_p(u):=\int_{\Rplus}|\nabla u(x)|^p\,dx-\left(\frac{p-1}{p}\right)^p\int_{\Rplus}\frac{|u(x)|^p}{x_d^p}\,dx.
$$
It turns out that this quantity can be estimated from below in various ways. For example, Hardy--Sobolev--Maz'ya inequality \cite{MR2863763,MR2424899} is the combination of Hardy and Sobolev inequalities and has the form
$$
I_p(u)\ge C_{d,p}\left(\int_{\Rplus}|u(x)|^qdx\right)^{p/q},
$$
where $2\le p<d$ and $q=\tfrac{dp}{d-p}$. For related results, see \cite{MR3736851}.

In this paper we are interested in another strengthening of \eqref{classicalHardyhalfspace}, that is the inequality
\begin{equation}\label{Hardyhalfspacewithremainder}
 \int_{\Rplus}|\nabla u(x)|^p\,dx\ge\left(\frac{p-1}{p}\right)^p\int_{\Rplus}\frac{|u(x)|^p}{x_d^p}\,dx+C(p,\tau)\int_{\Rplus}\frac{|u(x)|^p}{x_d^{p-\tau}\left(x_{d-1}^2+x_d^2\right)^{\tau/2}}\,dx,   
\end{equation}
where $u\in C_c^1\left(\Rplus\right)$, $C(p,\tau)$ is a positive constant and $\tau>0$. Inequality \eqref{Hardyhalfspacewithremainder} is sometimes referred to as Maz'ya inequality, or Hardy--Maz'ya inequality, and has a long history, dating back to 1970's. A special case of \eqref{Hardyhalfspacewithremainder} for $p=2$ and $\tau=1$ was obtained by Maz'ya \cite{Mazya1972} and later appeared in Maz'ya book \cite{MR2755141}. A generalised version was conjectured by Maz'ya and proved to be true by Tidblom for $0<\tau\le1$ \cite{MR2124873}. However, it is not known, what is the optimal value of $C(p,\tau)$, even for $p=2$ and the problem of finding the best constant in \eqref{Hardyhalfspacewithremainder} seems to be extremely challenging. Tidblom \cite[Corollary 3.1]{MR2124873} obtained the bound 
$$
C(p,1)\ge D(p)\left(\frac{p-1}{p}\right)^{p-1},
$$
where 
$$
D(p)=\begin{cases}
\displaystyle\frac{2}{2+3p},\quad 1<p<2,\\
\displaystyle\frac{1}{4(p-1)},\quad p\ge 2,
\end{cases}
$$
and, more generally, 
\begin{equation}\label{Cptau}
C(p,\tau)\ge D(p,\tau)\left(\frac{p-1}{p}\right)^{p-1},
\end{equation}
where
$$
D(p,\tau)=\begin{cases}
\displaystyle\frac{\tau^2}{2(1+p\tau^2)},\quad 1<p<2,\\
\displaystyle\frac{\tau^2}{2(p-1)(1+2\tau^2)},\quad p\ge 2
\end{cases}
$$
(notice that $D(p,1)\neq D(p)$). For example, $C(2,1)\ge\tfrac{1}{8}$.

Later, Maz'ya and Shaposhnikova \cite{MR2508844} showed that the largest value of $\Lambda$ in the inequality 
\begin{equation}\label{mazyashaposhnikovaC(2,1)}
 \int_{\Rplus}|\nabla u(x)|^2\,dx\ge\frac{1}{4}\int_{\Rplus}\frac{|u(x)|^2}{x_d^2}\,dx+\Lambda\int_{\Rplus}\frac{|u(x)|^2}{x_d\left(x_{d-1}^2+x_d^2\right)^{1/2}}\,dx
\end{equation}
is given by 
\begin{equation}\label{Lambda}
\Lambda=\inf\frac{\displaystyle\int_{0}^{\pi}\left[y'(t)^2+\frac{1}{4}y^2(t)\right]\sin t\,dt}{\displaystyle\int_{0}^{\pi}y^2(t)\,dt}\approx 0.1564,
\end{equation}
where the infimum is taken over all smooth functions on $[0,\pi]$, not constantly equal to zero. Thus, the problem of determining the exact value of the optimal constant $C(p,\tau)$ in \eqref{Hardyhalfspacewithremainder} in the case of $p=2$ requires knowledge on the first eigenvalue of certain differential operators. The inequality \eqref{mazyashaposhnikovaC(2,1)} arises in studies concerning degeneration in the oblique derivative problem for second order elliptic differential operators. In the same paper, the authors also consider more general expressions on the right-hand side of \eqref{mazyashaposhnikovaC(2,1)}.

A version of the inequality \eqref{Hardyhalfspacewithremainder} for half-space in the Heisenberg group was obtained in \cite[Theorem 1.1]{MR2442182}.
\subsection{Fractional Hardy inequalities}
Let us now briefly introduce basic facts on fractional Hardy inequalities and fractional Sobolev spaces. For $0<s<1$, $p\geq 1$ and $\Omega\subset\R^d$ being an open set, the latter is defined as
\begin{equation}\label{Wsp}
W^{s,p}(\Om)=\left\{u\in L^p(\Om)\colon\int_{\Om}\int_{\Om}\frac{|u(x)-u(y)|^p}{|x-y|^{d+sp}}\,dy\,dx<\infty\right\}.
\end{equation}
The space $W^{s,p}(\Om)$ is a Banach space endowed with the norm 
$$
\|u\|_{W^{s,p}(\Om)}=\|u\|_{L^p(\Om)}+\left(\int_{\Om}\int_{\Om}\frac{|u(x)-u(y)|^p}{|x-y|^{d+sp}}\,dy\,dx\right)^{\frac{1}{p}}.
$$
In this context, fractional Hardy inequality reads
\begin{equation}\label{generalfractionalHardy}
 \int_{\Om}\int_{\Om}\frac{|u(x)-u(y)|^p}{|x-y|^{d+sp}}\,dy\,dx\ge C_{d,s,p,\Om}\int_{\Om}\frac{|u(x)|^p}{\dd(x)^{sp}}\,dx,\quad u\in C_c(\Om).   
\end{equation}

Here and elsewhere, $\dd(x)=\inf_{y\in\partial\Om}|x-y|$ denotes distance to the boundary. Dyda and Vähäkangas \cite{MR3237044} proved that \eqref{generalfractionalHardy} holds under some assumptions on parameters $s,p,d$ and regularity of $\Omega$. This inequality is closely related to the density property of smooth, compactly supported functions in $W^{s,p}(\Om)$, see \cite{MR4454384}. Sharp constants for \eqref{generalfractionalHardy} are known for some special choices of the underlying domain. For $\Om=\R^d\setminus\{0\}$ and $p=2$, the sharp constant in \eqref{generalfractionalHardy} was found independently by Herbst \cite{MR436854}, Beckner \cite{MR1254832} and Yafaev \cite{MR1717839}. The case $p\ge 1$ was resolved by Frank and Seiringer \cite{MR2469027}. For the case of $\Om=\Rplus$, the best constant in \eqref{generalfractionalHardy} was obtained by Bogdan and Dyda \cite{MR2663757} for $p=2$ and Frank and Seiringer in \cite{MR2723817} for general $p\geq 1$. Loss and Sloane \cite{MR2659764} showed that \eqref{generalfractionalHardy} holds for all convex, proper domains in $\R^d$ with the same optimal constant as for the half-space. This constant takes the form 
\begin{equation}\label{Dsp}
\mathcal{D}_{d,s,p}=2\pi^{\frac{d-1}{2}}\frac{\Gamma\left(\frac{1+sp}{2}\right)}{\Gamma\left(\frac{d+sp}{2}\right)}\int_{0}^{1}\frac{\left|1-t^{(sp-1)/p}\right|^p}{(1-t)^{1+sp}}\,dt.    
\end{equation}
Explicitly, the following sharp fractional Hardy inequality holds,
\begin{equation}\label{sharpfractionalHardyhalfspace}
 \int_{\Rplus}\int_{\Rplus}\frac{|u(x)-u(y)|^p}{|x-y|^{d+sp}}\,dy\,dx\ge\mathcal{D}_{d,s,p}\int_{\Rplus}\frac{|u(x)|^p}{x_d^{sp}}dx,
\end{equation}
for $0<s<1$, $p\ge 1$, $u\in C_c\left(\Rplus\right)$ if $sp>1$ or $u\in C_c\left(\overline{\Rplus}\right)$ if $sp<1$.

Recent developments of \eqref{generalfractionalHardy} include weighted cases \cite{MR4705882}, inequalities with the distance to any flat submanifold \cite{MR5015195,MR4849884, KijaczkoSahu2026}, critical cases \cite{MR5026388,MR5018356}, Sobolev--Slobodeckij spaces \cite{MR4800921}, Orlicz spaces \cite{MR4815911}, Heisenberg group \cite{MR3807591,MR4908058} or Triebel--Lizorkin spaces \cite{Kijaczko2025arxiv}. In particular, the weighted fractional Hardy inequality
\begin{equation}\label{weightedsharpfractionalHardyhalfspace}
 \int_{\Rplus}\int_{\Rplus}\frac{|u(x)-u(y)|^p}{|x-y|^{d+sp}}x_d^{\al}y_d^{\be}\,dy\,dx\ge\mathcal{D}_{d,s,p,\al,\be}\int_{\Rplus}\frac{|u(x)|^p}{x_d^{sp-\al-\be}}dx,
\end{equation}
where $\al,\be\in(-1,sp)$ and $\al+\be>-1$, $0<s<1$, $p\ge 1$, $u\in C_c\left(\Rplus\right)$ if $sp-\al-\be>1$ and $u\in C_c\left(\overline{\Rplus}\right)$, if $sp-\al-\be<1$, is given in \cite[Theorem 1]{MR4705882}, or \cite[Theorem 2]{Kijaczko2025arxiv}. Here
\begin{equation}\label{Dspalbe}
\mathcal{D}_{d,s,p,\al,\be}=\pi^{\frac{d-1}{2}}\frac{\Gamma\left(\frac{1+sp}{2}\right)}{\Gamma\left(\frac{d+sp}{2}\right)}\int_{0}^{1}\frac{\left(t^{\al}+t^{\be}\right)\left|1-t^{(sp-\al-\be-1)/p}\right|^p}{(1-t)^{1+sp}}\,dt    
\end{equation}
is the sharp constant in \eqref{weightedsharpfractionalHardyhalfspace}. We will abbreviate $\mathcal{D}_{d,s,p,\al}:=\mathcal{D}_{d,s,p,\al,\al}$.

Our main result, Theorem \ref{thm1}, concerns a lower bound, analogous to \eqref{Hardyhalfspacewithremainder}, for the fractional Hardy difference
$$
I_{s,p}(u):= \int_{\Rplus}\int_{\Rplus}\frac{|u(x)-u(y)|^p}{|x-y|^{d+sp}}\,dy\,dx-\mathcal{D}_{d,s,p}\int_{\Rplus}\frac{|u(x)|^p}{x_d^{sp}}dx.
$$
In contrast to the local case, the literature concerning inequalities involving the quantity $I_{s,p}(u)$ is comparatively sparse. It is however known that the fractional counterpart of the Hardy--Sobolev--Maz'ya inequality, that is
\begin{equation}\label{HSM}
I_{s,p}(u)\ge C_{d,s,p}\left(\int_{\Rplus}|u(x)|^q\,dx\right)^{p/q},\quad q=\frac{dp}{d-sp},
\end{equation}
holds for $1<p<\tfrac{d}{s}$, see \cite{MR3803664, MR2910984, MR4708667, KijaczkoSahu2026,MR2823046}. For fractional Hardy inequalities on bounded domains with remainder terms, we refer to \cite{MR2755892} or to recent \cite{DiebTemgoua2026arxiv}. 
\subsection{Sobolev--Bregman forms}
For $p=2$, the Hardy inequality \eqref{generalfractionalHardy} can be generalised, using the notion of the so-called \emph{Sobolev--Bregman form}. The latter expression is defined as
$$
E_p[u]=\int_{\Om}\int_{\Om}\frac{\left(u(x)-u(y)\right)(u(x)^{\langle p-1\rangle}-u(y)^{\langle p-1\rangle})}{|x-y|^{d+\al}}\,dy\,dx,
$$
where $p>1$, $0<\al<2$ and  $a^{\langle t \rangle}:=|a|^{t}\text{sgn}(a)$ is the \emph{French power}. The sharp constant in the Hardy-type inequality
\begin{equation}\label{HardySobolevBregman}
E_p[u]\ge C_{d,\al,p,\Om}\int_{\Om}\frac{|u(x)|^p}{\dd(x)^{\al}}\,dx
\end{equation}
was found in \cite{MR4372148} for the whole space and in \cite{MR4720167} for half-space and convex domains. Notably, by \cite[Theorem 1.1]{MR4720167}, the following Hardy-type inequality holds for $u\in C_c\left(\Rplus\right)$,
\begin{equation}\label{HardySobolevBregmanhalfspace}    
\int_{\Rplus}\int_{\Rplus}\frac{\left(u(x)-u(y)\right)(u(x)^{\langle p-1\rangle}-u(y)^{\langle p-1\rangle})}{|x-y|^{d+\al}}\,dy\,dx\ge \mathcal{D}'_{d,\al,p}\int_{\Rplus}\frac{|u(x)|^p}{x_d^{\al}}\,dx,
\end{equation}
where 
\begin{equation}\label{Ddalp'}
\mathcal{D}'_{d,\al,p}=-\frac{2\pi^{\frac{d-1}{2}}\Gamma\left(\tfrac{1+\al}{2}\right)}{\Gamma\left(\tfrac{d+\al}{2}\right)}\left[\B\left(\tfrac{\al-1}{p}+1,-\al\right)+\B\left(\al-\tfrac{\al-1}{p},-\al\right)+\frac{1}{\al}\right]
\end{equation}
is the sharp constant (here $\B$ denotes the Euler Beta function).

Sobolev--Bregman forms are of independent interest, as they seem to serve as a natural tool to study $L^p$ properties of nonlocal operators. For example, in \cite{MR4372148} they were used to characterise the $L^p$ contractivity of the Feynman--Kac semigroup generated by the fractional Laplacian perturbed by Hardy potential. For their various applications and occurrences, we refer, for example, to \cite{MR4589708,MR4851904,MR4885983} and references therein.

Similarly as in the case of Gagliardo seminorm, we will be particularly interested in bounding from below the Hardy-type difference
$$
J_{\al,p}(u):=\int_{\Rplus}\int_{\Rplus}\frac{\left(u(x)-u(y)\right)(u(x)^{\langle p-1\rangle}-u(y)^{\langle p-1\rangle})}{|x-y|^{d+\al}}\,dy\,dx-\mathcal{D}'_{d,\al,p}\int_{\Rplus}\frac{|u(x)|^p}{x_d^{\al}}\,dx.
$$
\subsection{Main results}
\begin{thm}\label{thm1}
 Let $0<s<1$, $p>1$, and $d\ge 2$. Moreover, let $0<\tau<\tfrac{p(sp+1)}{2(p-1)}$, if $p\ge 2$ and $0<\tau<2\min\{sp,1\}$, if $1<p<2$. There exists a positive constant $C_{d,s,p,\tau}$ such that for all $u\in C_c\left(\Rplus\right)$ if $sp\ge 1$ or $u\in C_c\left(\overline{\Rplus}\right)$, if $sp<1$, 
 \begin{equation}\label{mainresult1}
 \int_{\Rplus}\int_{\Rplus}\frac{|u(x)-u(y)|^p}{|x-y|^{d+sp}}dy\,dx\ge\mathcal{D}_{d,s,p}\int_{\Rplus}\frac{|u(x)|^p}{x_d^{sp}}dx+C_{d,s,p,\tau}\int_{\Rplus}\frac{|u(x)|^p}{x_{d}^{sp-\tau}\left(x_{d-1}^2+x_d^2\right)^{\tau/2}}dx,
 \end{equation}
 where $\mathcal{D}_{d,s,p}$ is the sharp fractional Hardy constant given by \eqref{Dsp}.
\end{thm}
\color{black}
\begin{thm}\label{thm2}  Let $0<\alpha<2$, $p>1$, $d \ge 2$, and $0<\tau<1+\al$. There exists a positive constant $C'_{d,\alpha,p,\tau}$ such that for all $u\in C_c\left(\Rplus\right)$ if $\alpha\ge 1$ or $u\in C_c\left(\overline{\Rplus}\right)$, if $\alpha<1$, 
    \begin{equation}\label{mainresult2}
    \begin{split}
        \int_{\Rplus}\int_{\Rplus}\frac{\left(u(x)-u(y)\right)(u(x)^{\langle p-1\rangle}-u(y)^{\langle p-1\rangle})}{|x-y|^{d+\al}}\,dy\,dx &\geq \mathcal{D}_{d,\alpha,p}'\int_{\Rplus}\frac{|u(x)|^p}{x_d^{\alpha}}dx
        \\
        &+C'_{d,\alpha, p, \tau}\int_{\Rplus}\frac{|u(x)|^p}{x_{d}^{\al-\tau}\left(x_{d-1}^2+x_d^2\right)^{\tau/2}}dx,
    \end{split}
    \end{equation}
    where $\mathcal{D}_{d,\alpha,p}'$ is the sharp constant given by \eqref{Ddalp'}.
\end{thm}
The constants $C_{d, s, p, \tau}$ and $C'_{d, \alpha, p, \tau}$ are explicit, although of rather complicated form, and can be found in the proofs. However, there is no reason to believe that these constants are sharp. We suppose that the problem of determining the sharp constants in \eqref{mainresult1} and \eqref{mainresult2} goes way beyond methods used in this article and requires completely new ideas. As mentioned in the introduction, this problem is already very complicated in the local case.  Moreover, it is not clear if the assumptions regarding the range of the parameter $\tau$ in Theorems \ref{thm1} and \ref{thm2} are optimal and we could not manage to answer this question in this paper. 

We would like to emphasize that for $sp=1$ in \eqref{sharpfractionalHardyhalfspace} and for $\al=1$ in \eqref{HardySobolevBregmanhalfspace}, the constants $\mathcal{D}_{d,s,p}$ and $\mathcal{D}'_{d,\al,p}$ vanish, that is the corresponding Hardy inequalities become trivial. However, one may check that the constants $C_{d,s,p,\tau}$ and $C'_{d,\al,p,\tau}$ are positive also in these cases. In consequence, \eqref{mainresult1} and \eqref{mainresult2} are nontrivial.

\section{Proofs}
\subsection{Ground-state representation}
Let us recall basic concepts on the so-called \emph{ground state representation}, described by Frank and Seiringer in \cite{MR2469027}. This tool is available in both local and nonlocal setting. Thus, let us treat the local case first. Let $\Om\subset\R^d$ be open and nonempty. By \cite[(2.12)]{MR2469027}, if $g$ is nonnegative, $\om$ is positive and satisfies the weighted Euler--Lagrange equation in $\Om$,
\begin{equation}\label{V}
V(x)\om(x)^{p-1}=-\text{div}\left(g(x)|\nabla \om(x)|^{p-2}\nabla \om(x)\right),
\end{equation}
then, for $p\ge 2$, the following weighted Hardy inequality with a remainder,
\begin{equation}\label{generalweightedHardy}
\int_{\Om}|\nabla u(x)|^pg(x)\,dx\ge\int_{\Om}|u(x)|^pV(x)\,dx+c_p\int_{\Om}\left|\nabla v(x)\right|^p \om(x)^{p}g(x)\,dx,
\end{equation}
holds for any compactly supported $u$ satisfying $\int_{\Omega} |u(x)|^{p} V_{+}(x) \, dx < \infty$, where $V_+=\max\{0,V\}$. Here $v=u/\om$ and
\begin{equation}\label{cp}
c_p=\min_{0<\xi<\frac{1}{2}}\left((1-\xi)^p-\xi^p+p\xi^{p-1}\right).
\end{equation}
When $p=2$, \eqref{generalweightedHardy} becomes an equality with $c_2=1$. For $1<p<2$, the remainder term is more complicated and is given by
$$
\int_{\Om}|\nabla v(x)|^2\om(x)^2\left(|\nabla v(x)|\om(x)+|v(x)||\nabla \om(x)|\right)^{p-2}g(x)\,dx,
$$
see \cite[Lemma 2.2]{MR2400106}.

Let us now move to the fractional setting. Following again the framework in \cite{MR2469027}, consider a symmetric, nonnegative, measurable kernel $k(x,y)$ and a nonempty open set $\Omega \subset \mathbb{R}^{d}$. We define the functional
\begin{equation}
E[u] := \int_{\Omega} \int_{\Omega} |u(x)-u(y)|^{p} k(x,y) \, dy \, dx
\end{equation} 
and 
\begin{equation}
V_{\varepsilon} (x) := 2 \omega(x)^{-p+1} \int_{\Omega} (\omega(x)-\omega(y)) |\omega(x)-\omega(y)|^{p-2} k_{\varepsilon}(x,y) \, dy,
\end{equation}
where $\omega$ is a positive, measurable function on $\Omega$ and $\{k_{\varepsilon}(x,y) \}_{\varepsilon>0}$ is a family of measurable, symmetric kernels satisfying the assumptions $0 \leq k_{\varepsilon}(x,y) \leq k(x,y)$, $\lim_{\varepsilon \to 0} k_{\varepsilon}(x,y) = k(x,y)$ for almost all $x,y \in \Omega$. Assuming the integrals that define $V_{\varepsilon}$ are absolutely convergent for almost every $x \in \Omega$ and that $V_{\varepsilon}$ converges weakly in $L^{1}_{\text{loc}}(\Omega)$ to some function $V$, as $\ve\rightarrow 0^+$, for $p\ge 2$ we have the following Hardy-type inequality:
\begin{equation}\label{General Hardy inequality}
E[u] \geq \int_{\Omega} |u(x)|^{p}V(x) \, dx+c_p E_{\om}[u],
\end{equation}
for any compactly supported $u$ satisfying $\int_{\Omega} |u(x)|^{p} V_{+}(x) \, dx < \infty$, see \cite[Proposition $2.2$]{MR2469027}. Here $c_p$ is given by \eqref{cp} and the remainder term is
$$
E_{\om}[u]=\int_{\Om}\int_{\Om}\left|v(x)-v(y)\right|^p k(x,y)\om(x)^{p/2}\om(y)^{p/2}\,dy\,dx,\quad v=\tfrac{u}{\om}.
$$
As in the local case, \eqref{General Hardy inequality} becomes an equality for $p=2$. The remainder in the case $1<p<2$ was found by Dyda and Kijaczko in \cite{MR4708667}  (see also \cite{MR4597627}). Denoting
\begin{equation*} \widetilde{E}_{\omega}[u]:=\int_{\Omega}\int_{\Omega}\left(u(x)^{\langle p/2 \rangle }-u(y)^{\langle p/2 \rangle }\right)^2W(x,y)k(x,y)\,dy\,dx,
\end{equation*}
where 
\begin{equation*}
    W(x,y):=\min\{\omega(x),\omega(y)\}\max\{\omega(x),\omega(y)\}^{p-1}=\omega(x)\omega(y)\max\{\omega(x),\omega(y)\}^{p-2},
\end{equation*}
we have 
\begin{equation}\label{General Hardy with weight1<p<2}
E[u] - \int_{\Omega} |u(x)|^{p}V(x) \, dx \geq C_{p} \widetilde{E}_{\omega}[v] , \hspace{3mm} \quad v=\tfrac{u}{\om},
\end{equation} 
where $C_{p}=\max\left\{\tfrac{p-1}{p},\tfrac{p(p-1)}{2}\right\}$. When $u$ is a nonnegative function, $C_p$ can be taken as $p-1$.

Nonlocal ground-state representation is also available in the setting of Sobolev--Bregman forms, see \cite{MR4720167}. Let $p>1$. We keep the same assumptions as above, concerning the kernels $k_{\varepsilon}$ and $k$. Let 
$$
E_p[u]=\int_{\Om}\int_{\Om}(u(x)-u(y))(u(x)^{\langle p-1\rangle}-u(y)^{\langle p-1\rangle}) k(x,y)\,dy\,dx
$$
and 
$$
V_{\ve}(x)=2\int_{\Om}\left(\frac{1}{p}\frac{\om(x)^{p-1}-\om(y)^{p-1}}{\om(x)^{p-1}}+\frac{p-1}{p}\frac{\om(x)-\om(y)}{\om(x)}\right)k_{\ve}(x,y)\,dy.
$$
We assume again that $V_{\ve}$ are absolutely convergent and $V_{\ve}\rightarrow V$ weakly in $L^{1}_{\text{loc}}(\Omega)$. Moreover, for $a,b\in\R$, let us define the so-called \emph{Bregman divergence}
$$
F_p(a,b):=|b|^p-|a|^p-pa^{\langle p-1\rangle}(b-a).
$$
The function $F_p$ is always nonnegative, as the second-order Taylor remainder of the convex function $x\mapsto |x|^p$. Then, for compactly supported $u$ such that $\int_{\Omega} |u(x)|^{p} V_{+}(x) \, dx < \infty$, the following identity holds,
\begin{equation}\label{groundstateSobolevBregman}
E_p[u]=\int_{\Om}|u(x)|^pV(x)\,dx+\frac{2}{p}\int_{\Om}\int_{\Om}F_p\left(\frac{u(x)}{\om(x)},\frac{u(y)}{\om(y)}\right)\om(x)^{p-1}\om(y)k(x,y)\,dy\,dx.    
\end{equation}
In particular, we have a Hardy-type inequality
\begin{equation}\label{HardySB}
E_p[u]\ge\int_{\Om}|u(x)|^pV(x)\,dx.
\end{equation}
\begin{rem}\label{rem3}
    The assumption that $V_{\ve}\rightarrow V$ weakly in $L^1_{\text{loc}}(\Om)$ is not necessary for \eqref{General Hardy inequality}, \eqref{General Hardy with weight1<p<2} and \eqref{HardySB} to hold. It is needed to ensure that 
    $$\lim_{\ve\rightarrow 0^+}\int_{\Om}|u(x)|^p V_{\ve}(x)\,dx=\int_{\Om}|u(x)|^p V(x)\,dx,$$ 
    however, we are satisfied with
    $$\liminf_{\ve\rightarrow 0^+}\int_{\Om}|u(x)|^p V_{\ve}(x)\,dx\ge \int_{\Om}|u(x)|^p V(x)\,dx.$$
Therefore, instead of weak convergence, we can assume, for example, one of the following conditions:
\begin{align*}
(i)&\quad V_{\ve}(x)\ge V(x)\text{ for almost all }x\in\Omega;\\
(ii)&\quad V_{\ve}(x)\ge \widetilde{V}_{\ve}(x)>0\text{  and } \widetilde{V}_{\ve}(x)\rightarrow V(x) \text{ almost everywhere in }\Om.     
\end{align*}

Indeed, the proofs of \eqref{General Hardy inequality}, \eqref{General Hardy with weight1<p<2} (perhaps with inequality instead of equality for $p=2$) and \eqref{HardySB} are straightforward in the case (i), while in the case (ii) follows from Fatou's lemma.
\end{rem}
\subsection{The local case}
We will now present a way to derive inequality \eqref{Hardyhalfspacewithremainder}, using ground-state transforms. The proof is different from the one of Tidblom \cite{MR2124873}, who uses a vector field approach. Since we aim for a generalisation to the fractional setting, we need to develop a different method, because the vector field approach is not applicable here. Nevertheless, our proof is based on the following simple trick, which is also used in \cite{MR2124873}: having two estimates of the form
$$
X\ge aY,\quad X\ge-bY+Z,\quad X,Y,Z,a,b>0,
$$
we combine it into $X\ge\tfrac{a}{a+b}Z$, in order to get rid of $Y$ and obtain a relation only between $X$ and $Z$.

Let $p\ge 2$ and let $u\in C_c^1\left(\R^d_+\right)$. By the ground-state representation \eqref{generalweightedHardy} applied to $g(x)\equiv1$ and $\om(x)=x_d^{(p-1)/p}$, we have 
\begin{equation}\label{hardywithremainderhalfspace}
\int_{\Rplus}|\nabla u(x)|^p\,dx\ge\left(\frac{p-1}{p}\right)^p\int_{\Rplus}\frac{|u(x)|^p}{x_d^p}\,dx+c_p\int_{\Rplus}|\nabla v(x)|^px_{d}^{p-1}\,dx,
\end{equation}
where $v(x)=u(x)x_d^{(1-p)/p}$ and $c_p$ is given by \eqref{cp}. We will now estimate the remainder term. Let $g(x)=x_d^{p-1}$ and $\om(x)=(x_{d-1}^2+x_d^2)^{\ga/2}$, $\ga\in\R$. We will again use \eqref{generalweightedHardy} for this new choice of $g$ and $\om$. Then, we may compute that the potential $V$ given by \eqref{V} is 
\begin{equation}\label{W}
V(x)=V_{\ga}(x)=-C_p(\ga)\frac{x_d^{p-1}}{\left(x_{d-1}^2+x_d^2\right)^{p/2}},
\end{equation}
with 
$$
C_p(\ga)=\ga|\ga|^{p-2}\left(1+\ga(p-1)\right)
$$
The choice $\ga_0=-\tfrac{1}{p}$ maximises the constant $-C_p(\ga)$ and we get 
$$
V_{\ga_0}(x)=\frac{1}{p^p}\frac{x_d^{p-1}}{\left(x_{d-1}^2+x_d^2\right)^{p/2}}.
$$
This yields a Hardy inequality
\begin{align}\label{Hardy1}
\int_{\Rplus}|\nabla u(x)|^p\,dx&\ge\left(\frac{p-1}{p}\right)^p\int_{\Rplus}\frac{|u(x)|^p}{x_d^p}\,dx+\frac{c_p}{p^p}\int_{\Rplus}\frac{|u(x)|^p}{\left(x_{d-1}^2+x_d^2\right)^{p/2}}\,dx.
\end{align}
Using the ground-state representation again, for $\ga>0$, from \eqref{hardywithremainderhalfspace} and \eqref{W} we derive another Hardy inequality (applied to the remainder term $\int_{\Rplus}|\nabla v(x)|^px_{d}^{p-1}\,dx$) of the form
\begin{align}\label{Hardy2}
\nonumber\int_{\Rplus}|\nabla u(x)|^p\,dx&\ge\left(\frac{p-1}{p}\right)^p\int_{\Rplus}\frac{|u(x)|^p}{x_d^p}\,dx-c_pC_p(\ga)\int_{\Rplus}\frac{|u(x)|^p}{\left(x_{d-1}^2+x_d^2\right)^{p/2}}dx\\
&\quad+c_p^2\int_{\Rplus}|\nabla\widetilde{v}(x)|^px_{d}^{p-1}\left(x_{d-1}^2+x_d^2\right)^{p\ga/2}\,dx,
\end{align}
where $\widetilde{v}(x)=v(x)\left(x_{d-1}^2+x_d^2\right)^{-\ga/2}=u(x)x_d^{(1-p)/p}\left(x_{d-1}^2+x_d^2\right)^{-\ga/2}$. Thus, combining \eqref{Hardy1} with \eqref{Hardy2}, we arrive at the inequality
\begin{align*}
\nonumber\int_{\Rplus}|\nabla u(x)|^p\,dx&\ge\left(\frac{p-1}{p}\right)^p\int_{\Rplus}\frac{|u(x)|^p}{x_d^p}dx\\
&\quad+\frac{c_p^2}{1+C_p(\ga)p^p}\int_{\Rplus}|\nabla\widetilde{v}(x)|^px_{d}^{p-1}\left(x_{d-1}^2+x_d^2\right)^{p\ga/2}dx.    
\end{align*}
Recall the classical weighted Hardy inequality on the half-space, that is 
\begin{equation}\label{weightedclassicalHardyhalfspace}
\int_{\Rplus}|\nabla f(x)|^px_d^{\al}\,dx\ge\left|\frac{p-1-\al}{p}\right|^p\int_{\Rplus}|f(x)|^p x_d^{\al-p}\,dx
\end{equation}
valid for $f\in C_c^1\left(\Rplus\right)$ and $\al\in\R$. Using \eqref{weightedclassicalHardyhalfspace}, we can estimate
\begin{align*}
\int_{\Rplus}|\nabla\widetilde{v}(x)|^px_{d}^{p-1}\left(x_{d-1}^2+x_d^2\right)^{p\ga/2}\,dx&\ge\int_{\Rplus}|\nabla\widetilde{v}(x)|^px_{d}^{p-1+p\ga}\,dx\\
&\ge\left|\frac{p-1-(p-1+p\ga)}{p}\right|^p\int_{\Rplus}|\widetilde{v}(x)|^px_d^{-1+p\ga}\,dx\\
&=\ga^p\int_{\Rplus}\frac{|u(x)|^p}{x_d^{p-p\ga}\left(x_{d-1}^2+x_d^2\right)^{p\ga/2}}\,dx.
\end{align*}
Hence, writing now $p\ga=\tau$ and summarising all the results above, we obtain the inequality
\begin{equation}\label{mazya}
\int_{\Rplus}|\nabla u(x)|^p\,dx\ge\left(\frac{p-1}{p}\right)^p\int_{\Rplus}\frac{|u(x)|^p}{x_d^p}dx+\widetilde{C}(p,\tau)\int_{\Rplus}\frac{|u(x)|^p}{x_d^{p-\tau}\left(x_{d-1}^2+x_d^2\right)^{\tau/2}}\,dx,   
\end{equation}
where $\tau>0$ and 
$$
\widetilde{C}(p,\tau)=\frac{c_p^2\tau^p}{p^p\left(1+p^pC_p\left(\tfrac{\tau}{p}\right)\right)}=\frac{c_p^2\tau^p}{p^p\left(1+p\tau^{p-1}+(p-1)\tau^p\right)}.
$$
Needless to say, this procedure results, in general, in worse constants than the constants obtained by Tidblom in \cite{MR2124873} for $0<\tau\le 1$. For example, $\widetilde{C}(2,1)=\tfrac{1}{16}$, as originally obtained by Maz'ya, and we know that it can be increased to $\tfrac{1}{8}$, or even to \eqref{Lambda}. However, it works for all $\tau>0$, while in Tidblom's article \cite{MR2124873} the assumption $0<\tau\le 1$ is stated. Moreover, the presented approach does not seem to apply to the case of $1<p<2$ (although it will work in the fractional setting).

We can easily generalise inequality \eqref{mazya} by considering
$$
\om(x)=\left(x_{d-k}^2+\ldots+x_d^2\right)^{\gamma/2},\quad 1\leq k\le d-1.
$$
It then holds
$$
V(x)=-|\gamma|^{p-2}\ga\left(k+(p-1)\ga\right)\frac{x_d^{p-1}}{\left(x_{d-k}^2+\ldots+x_d^2\right)^{p/2}}
$$
and the choice $\ga_0=-\frac{k}{p}$ maximises the constant in front of the above expression. Similar calculations as before lead to the inequality
\begin{equation}\label{mazyak}
\int_{\Rplus}|\nabla u(x)|^p\,dx\ge\left(\frac{p-1}{p}\right)^p\int_{\Rplus}\frac{|u(x)|^p}{x_d^p}dx+C_k(p,\tau)\int_{\Rplus}\frac{|u(x)|^p}{x_d^{p-\tau}\left(x_{d-k}^2+\ldots+x_d^2\right)^{\tau/2}}\,dx,   
\end{equation}
where
$$
C_k(p,\tau)=\frac{c_p^2(k\tau)^p}{p^p\left(k^p+pk\tau^{p-1}+(p-1)\tau^p\right)}.
$$

\subsection{Weighted fractional $p$-Laplacian for power functions}
Our goal now is to extend the above reasoning to the fractional case. This will of course be much more complicated, as we will deal with nonlocal operators. To this end, in this subsection we will calculate or estimate a series of weighted fractional $p$-Laplacians acting on power functions, in order to obtain a fractional version of \eqref{W}. These results are of independent interest and might serve in the future as a tool to obtain various Hardy-type inequalities.

For $d \ge 2$ and $x = (x_1,...,x_d) \in \R^d$ we will write $\widetilde{x} := (x_{d-1}, x_{d}) \in \R^2$. For $d \ge 1$ and $x \in \R^d\setminus\{0\}$ we define $\hat{x} := \frac{x}{|x|}$. We also define $\Sd_+:=\Sd\cap\Rplus$ and $B^d_+:=B^d\cap\Rplus$. Here, $\Sd$ and $B^d$ denote the unit sphere and ball in $\R^d$, respectively.

The next lemma provides a fractional counterpart of the formula \eqref{W}. It will serve as a main tool in proving Theorem \ref{thm1} for $p\ge 2$. First of all, let us consider a family of kernels that will be helpful in deriving the ground state representation in the fractional case. For $d \geq 2$, $0 < \varepsilon < 1$ let $k_{d, \xi}^{(\varepsilon)} : \R^d \times \R^d \rightarrow [0, \infty)$ be defined as
$$k_{d, \xi}^{(\varepsilon)}(x, y) := \begin{cases}
    |x - y|^{-d-\xi}\mathds{1}_{D_\varepsilon}(|\widetilde{x}|, |\widetilde{y}|), \quad &x \neq y,
    \\
    0 \quad &x = y,
\end{cases}$$
where
\begin{align}\label{eq: D_eps}D_\varepsilon := \left\{(a, b) \in [0, \infty)^2 : a \leq b(1 - \varepsilon) \lor b \leq a(1 - \varepsilon) \right\}.
\end{align}
Observe that $D_\varepsilon$ is a cone in $[0, \infty)^2$, thus, we have 
$$k_{d, \xi}^{(\varepsilon)}(rx, ry) = r^{-d-\xi}k_{d, \xi}^{(\varepsilon)}(x, y),$$
for all $x, y \in \R^d$, and $r > 0$. Moreover, $D_{\varepsilon_2} \subset D_{\varepsilon_1}$ whenever $0 < \varepsilon_1 \leq \varepsilon_2 < 1$ and 
$$\bigcup_{0 < \varepsilon< 1}D_\varepsilon = [0, \infty)^2 \setminus \{(a, a) : a\in(0, \infty)].$$
This implies that 
$$\lim_{\varepsilon \rightarrow 0^+} k_{d, \xi}^{\varepsilon}(x, y) = |x - y|^{-d-\xi} \quad \text{a.e.}$$
\color{black} 

\begin{lem}\label{lemma: weightedLaplacian} 
Let $x\in\R^2_+$, $0<s<1$, $p > 1$, $\beta \in (-1, sp)$, and $\gamma \in \left( -\frac{\beta+2}{p-1}, \frac{sp-\beta}{p-1}\right)$. Then, uniformly on compact sets contained in $\R^2_+$,
    \begin{align*}
        \lim_{\varepsilon \rightarrow 0^+} &\int_{\R^2_+}(|x|^{\gamma} - |y|^{\gamma})||x|^\gamma - |y|^\gamma|^{p-2}y_2^{\beta}k_{2, sp}^{(\varepsilon)}(x, y)\,dy =|x|^{\gamma(p-1) + \beta - sp}F_{s, p, \beta, \gamma}(\hat{x}),
    \end{align*}
where $\va$ is the argument of $x$ and 
\begin{align}\label{F(x)}
&F_{s, p, \beta, \gamma}(\hat{x})\\
&\nonumber:=\sgn(\ga(sp - 2 - 2\beta - \gamma(p-1)))\int_{B^2_+}\frac{\left|1-|t|^{\ga}\right|^{p-1}\left|1-|t|^{sp - 2 - 2\beta - \gamma(p-1)}\right|}{|\hat{x}-t|^{sp+2}}t_2^{\be}\,dt\\
&\nonumber =\sgn(\ga(sp - 2 - 2\beta - \gamma(p-1)))\int_{0}^{1}\int_{0}^{\pi}\frac{r^{1+\be}\left|1-r^{\ga}\right|^{p-1}\left|1-r^{sp - 2 - 2\beta - \gamma(p-1)}\right|}{\left(1-2r\cos(\va-y)+r^2\right)^{\frac{sp+2}{2}}}(\sin y)^{\be}\,dy\,dr.
\end{align}
\end{lem}

\begin{proof} For $\varepsilon > 0$ we have
\begin{align*}
    A_\varepsilon := &\int_{\R^2_+}(|x|^{\gamma} - |y|^{\gamma})||x|^\gamma - |y|^\gamma|^{p-2}y_2^{\beta}k_{2, sp}^{(\varepsilon)}(x, y)\,dy
    \\
    = &|x|^{\gamma(p-1) + \beta - sp}\int_{\R_+^2}\frac{(1 - |t|^{\gamma})|1 - |t|^\gamma|^{p-2}}{|\hat{x} - t|^{sp +2}}t_2^{\beta} \mathds{1}_{D_\varepsilon}(1, |t|)dt.
\end{align*}
Observe that, from the definition \eqref{eq: D_eps} of $D_\varepsilon$, the indicator $\mathds{1}_{D_\varepsilon}(1, |t|)$ splits $A_\varepsilon$ into two integrals over 
$$\{ t \in \R^2_+ : |t| \leq 1 - \varepsilon \} \quad \text{ and } \quad \{ t \in \R^2_+ : |t| \geq (1 - \varepsilon)^{-1} \}.$$ 
It may be checked that the assumptions on $\ga$ and $\be$ guarantee convergence of both integrals. By the change of variables $t\mapsto\frac{t}{|t|^2}$ with the Jacobian $|t|^{-4}$, we have
\begin{align*}
    &\int_{\{t \in \R_+^2 : |t| \geq (1 - \varepsilon)^{-1}\}}\frac{(1 - |t|^{\gamma})|1 - |t|^\gamma|^{p-2}}{|\hat{x} - t|^{sp +2}}t_2^{\beta}\,dt.
    \\
    &= \int_{\{t \in \R_+^2 : |t| \leq 1 - \varepsilon\}}
    \frac{|t|^{sp - 2 - 2\beta - \gamma(p-1)}(|t|^{\ga} - 1)|1 - |t|^\gamma|^{p-2}}{|\hat{x} - t|^{sp +2}}t_2^{\beta}\,dt\\
\end{align*}
Thus, we get 
\begin{align*}
    A_\varepsilon = |x|^{\gamma(p-1) + \beta - sp} \int_{\{t \in \R_+^2 : |t| \leq 1-\varepsilon\}}\frac{(1 - |t|^{sp - 2 - 2\beta - \gamma(p-1)})(1 - |t|^\gamma)|1 - |t|^{\ga}|^{p-2}}{|\hat{x} - t|^{2 + sp}}t_2^\beta\,dt,
\end{align*}
Now, since the integrand in $A_\varepsilon$ does not change sign for $t \in B^2_+$, the desired result follows from the monotone convergence theorem.
\end{proof}

\begin{rem}\label{rem8}
For fixed $\be$, $s$ and $p$, it is elementary to check that the choice 
$$
\ga_0=\frac{sp-2-2\be}{p}
$$
maximises the function $\ga\mapsto F_{s,p,\be,\ga}(\hat{x})$, for fixed $s$, $p$, $\be$ and $\hat{x}$,  where $F_{s,p,\be,\ga}$ is given by \eqref{F(x)}, on an interval where $F_{s,p,\be,\ga}(\hat{x})$>0. 
\end{rem}

\begin{lem}\label{lemma: weighted plaplacian}
Let $0<s<1$, $d\ge 2$, $p > 1$, $\beta \in (-1, sp)$, and $\gamma \in \left( -\frac{\beta+2}{p-1}, \frac{sp-\beta}{p-1}\right)$. Let  $\om(x):=\left(x_{d-1}^2+x_d^2\right)^{\frac{\ga}{2}}=|\widetilde{x}|^{\ga}$. Then, 
\begin{align*}
\lim_{\ve\rightarrow 0^+}&\int_{\Rplus}\left(\om(x)-\om(y)\right)\left|\om(x)-\om(y)\right|^{p-2}y_d^{\be}k_{d, sp}^{(\varepsilon)}(x, y)\,dy=\mathcal{A}_{d, sp}|\widetilde{x}|^{\gamma(p-1) + \beta - sp}F_{s,p,\be,\ga}(\hat{x}),
\end{align*}
uniformly on compact sets contained in $\Rplus$. Here $F_{s,p,\be,\ga}(\hat{x})$ is given by \eqref{F(x)} and 
$$\mathcal{A}_{d, \xi} := \frac{\pi^{\frac{d-2}{2}}\Gamma\left(\frac{2+\xi}{2}\right)}{\Gamma\left(\frac{d+\xi}{2}\right)}, \quad \xi > -2.$$
\end{lem}
\begin{proof}
For $d\ge 3$, we use the identity 
\begin{equation}\label{integral identity}
\int_{\R^k}\frac{dz}{\left(|z|^2+a^2\right)^{\frac{\al}{2}}}=\frac{\pi^{\frac{k}{2}}\Gamma\left(\frac{\al-k}{2}\right)}{\Gamma\left(\frac{\al}{2}\right)}a^{k-\al} = \mathcal{A}_{k+2, \alpha-k-2}a^{k-\al},\quad a>0,\,\al>k,
\end{equation}
applied to $k=d-2$, $\al=d+sp$ and $a=|\widetilde{x}-\widetilde{y}|$. Hence, using Fubini's theorem (which we are allowed to apply, since the integrals appearing in the proof of Lemma \ref{lemma: weightedLaplacian} are absolutely convergent), after integrating over first $d-2$ variables,
\begin{align*}
\lim_{\ve\rightarrow 0^+}&\int_{\Rplus}\left(\om(x)-\om(y)\right)\left|\om(x)-\om(y)\right|^{p-2}y_d^{\be}k_{d, sp}^{(\varepsilon)}(x, y)\,dy\\
&=\mathcal{A}_{d, sp}\lim_{\ve\rightarrow 0^+}\int_{\R^2_+}\left(\om(x)-\om(y)\right)\left|\om(x)-\om(y)\right|^{p-2}y_d^{\be}k_{2, sp}^{(\varepsilon)}(\widetilde{x}, \widetilde{y})\, d\widetilde{y}
\end{align*}
and we can apply Lemma \ref{lemma: weightedLaplacian}.
\end{proof}
When $1<p<2$, the remainder term in the nonlocal ground-state representation \eqref{General Hardy with weight1<p<2} is of different, more complicated form than in the case of $p\ge 2$. Our next results contain estimates of fractional $p$-Laplacians with specific weight and will be needed to deal with the proof of Theorem \ref{thm1} for $1<p<2$.

We begin with two useful technical lemmas. By $\2F1(a,b,c;z)$ we denote the Gauss hypergeometric function.

\begin{lem}\label{integralonsphere}
Let $d\ge 2$ and $\xi > 0$. There exists a constant $C=C(d,s,p)>0$ such that for all $0<r<1$ and $\si\in\Sd$ it holds
\begin{equation}\label{estimat}
\int_{\Sd_+}\frac{d\om}{|\om-r\si|^{d+\xi}}\le \frac{C}{(1-r)^{1+\xi}}.    
\end{equation}
\end{lem}
\begin{proof}
Clearly, by rotation invariance,
$$
\int_{\Sd_+}\frac{d\om}{|\om-r\si|^{d+\xi}}\le \int_{\Sd}\frac{d\om}{|\om-r\si|^{d+\xi}}=\int_{\Sd}\frac{d\om}{\left(1-2r\om_d+r^2\right)^{\frac{d+\xi}{2}}}.
$$
By \cite[(3.665)]{MR2360010} and Euler formula for the hypergeometric function,
\begin{align*}
\int_{\Sd}\frac{d\om}{\left(1-2r\om_d+r^2\right)^{\frac{d+\xi}{2}}}&=\meas\2F1\left(\tfrac{d+\xi}{2},\tfrac{2+\xi}{2},\tfrac{d}{2};r^2\right)\\
&=\meas (1-r^2)^{-1-\xi}\2F1\left(-\tfrac{\xi}{2},\tfrac{d-2-\xi}{2},\tfrac{d}{2};r^2\right)\\
&\le C(1-r)^{-1-\xi},
\end{align*}
because the value $\2F1\left(-\tfrac{\xi}{2},\tfrac{d-2-\xi}{2},\tfrac{d}{2};1\right)$ is finite, since $\tfrac{d}{2}-\tfrac{d-2-\xi}{2}-\left(\tfrac{-\xi}{2}\right)=1+\xi>0$.
\end{proof}
\begin{lem}
\label{lem: bounds for H}
    Let $\be\ge0$, $0<r<1$, $\xi > 0$, and
    $$
    H_{r, \xi, \beta}(\va):=\int_0^{\pi}\frac{\left(\sin y\right)^{\be}}{\left(1-2r\cos(\va-y)+r^2\right)^{\frac{2+\xi}{2}}}\,dy=\int_0^{\pi}\frac{\left(\sin y\right)^{\be}}{\left(1+2r\cos(\va+y)+r^2\right)^{\frac{2+\xi}{2}}}\,dy.
    $$
Then, for $0\le\va\le\pi$ we have $H_{r, \xi, \beta}(0)\le H_{r, \xi, \beta}(\va)\le H_{r, \xi, \beta}(\tfrac{\pi}{2})$.
\end{lem}
\begin{proof}
For the simplicity let us write $H := H_{r, \xi, \beta}$. By a substitution $y\mapsto\pi-y$ it is easy to see that $H$ is symmetric on $[0,\pi]$ with respect to the midpoint. Therefore, it suffices to show that $H$ is increasing on $[0,\tfrac{\pi}{2}]$. For $\va\in(0,\tfrac{\pi}{2})$ we have
$$
H'(\va)=r(2+\xi)\int_{0}^{\pi}\frac{\left(\sin y\right)^{\be}\sin(\va+y)}{\left(1+2r\cos(\va+y)+r^2\right)^{\frac{4+\xi}{2}}}\,dy=:r(2+\xi)I(\va).
$$
We split the integration as follows,
\begin{align*}
I(\va)&=\left(\int_{0}^{\pi-2\va}+\int_{\pi-2\va}^{\pi-\va}+\int_{\pi-\va}^{\pi}\right)\frac{\left(\sin y\right)^{\be}\sin(\va+y)}{\left(1+2r\cos(\va+y)+r^2\right)^{\frac{4+\xi}{2}}}\,dy\\
&=:I_1(\va)+I_2(\va)+I_3(\va).
\end{align*}
Clearly, $I_1(\va)\ge0$. Substituting $y=\pi-\va-s$, $s\in(0,\va)$, we get
\begin{align*}
I_2(\va)=\int_{0}^{\varphi}\frac{\left(\sin(\pi-\va-s)\right)^{\be}\sin(\pi-s)}{\left(1-2r\cos s+r^2\right)^{\frac{4+\xi}{2}}}\,ds= \int_{0}^{\varphi}\frac{\left(\sin(\va+s)\right)^{\be}\sin s}{\left(1-2r\cos s+r^2\right)^{\frac{4+\xi}{2}}}\,ds.
\end{align*}
Moreover, substituting $y=\pi-\va+s$, $s\in (0,\varphi)$, we get
$$
I_3(\va)=\int_{0}^{\varphi}\frac{\left(\sin(\pi-\va+s)\right)^{\be}\sin(\pi+s)}{\left(1-2r\cos s+r^2\right)^{\frac{4+\xi}{2}}}\,ds=-\int_{0}^{\varphi}\frac{\left(\sin(\va-s)\right)^{\be}\sin s}{\left(1-2r\cos s+r^2\right)^{\frac{4+\xi}{2}}}\,ds.
$$
Thus,
$$
I_2(\va)+I_3(\va)=\int_{0}^{\va}\frac{\sin s\left[\left(\sin(\va+s)\right)^{\be}-\left(\sin(\va-s)\right)^{\be}\right]}{\left(1-2r\cos s+r^2\right)^{\frac{4+\xi}{2}}}\,ds.
$$
Since $\be\ge 0$, to justify that $I_2(\va)+I_3(\va)\ge 0$, we need to ensure that $\sin(\va+s)\ge \sin(\va-s)$. This is satisfied, because we have $\va+s,\va-s\in(0,\pi)$ and $|\va+s-\tfrac{\pi}{2}|<|\va-s-\tfrac{\pi}{2}|$, which is equivalent to $(\va-\tfrac{\pi}{2})s<0$, which is true. Hence, $I(\va)\ge 0$ and the proof is complete.
\end{proof}

\begin{rem}
When $-1<\be<0$, the assertion of Lemma \ref{lem: bounds for H} is no longer true. In this case, since $(1-r)^2\le1-2r\cos(\va-y)+r^2\le (1+r)^2$, we can bound $H_{r, \xi, \beta}(\va)$ in the non-optimal way, that is $(1+r)^{-2-\xi}\int_{0}^{\pi}(\sin y)^{\be}\,dy\le H_{r, \xi, \beta}(\va)\le (1-r)^{-2-\xi}\int_{0}^{\pi}(\sin y)^{\be}\,dy.$
Furthermore, since $\int_{0}^{\pi}(\sin y)^{\be}\,dy=\B\left(\tfrac{\be+1}{2},\tfrac{1}{2}\right)$, we have 
\begin{equation}\label{bounds for H}
\frac{\B\left(\tfrac{\be+1}{2},\tfrac{1}{2}\right)}{(1+r)^{2+\xi}}\le H_{r, \xi, \beta}(\va)\le\frac{\B\left(\tfrac{\be+1}{2},\tfrac{1}{2}\right)}{(1-r)^{2+\xi}}.
\end{equation}
\end{rem}

\begin{lem}\label{lem9}
 Let $1\le\al<2$ and $x\in\R^2_+$. Then\\
 (i) for $0<\ga<1$,
 \begin{align*}
  \liminf_{\ve\rightarrow 0^+}&\int_{\R^2_+}(|x|^{\ga}-|y|^{\ga})\min\left\{x_2^{\al-1},y_{2}^{\al-1}\right\}k_{2, \alpha}^{(\varepsilon)}(x, y) \, dy\ge -C^{(1)}_{\al, \ga}|x|^{\ga-\al}x_2^{\al-1},
 \end{align*}
 where
 \begin{equation}\label{C1(al,ga)}
 C^{(1)}_{\al, \ga} :=\int_{0}^{1}\int_{0}^{\pi}\frac{r^{\al}(1-r^{\ga})(r^{-\ga-\al}-1)}{\left(1-2r\sin\theta+r^2\right)^{\frac{2+\al}{2}}}d\theta\,dr>0;
 \end{equation}
 \\
 (ii) if $\al-2<\ga<0$, then
 \begin{align*}
  \liminf_{\ve\rightarrow 0^+}&\int_{\R^2_+}(|x|^{\ga}-|y|^{\ga})\min\left\{x_2^{\al-1},y_{2}^{\al-1}\right\}k_{2, \alpha}^{(\varepsilon)}(x, y)dy\ge C^{(2)}_{\al, \ga}|x|^{\ga-\al}x_2^{\al-1},
 \end{align*}
 where
\begin{equation}\label{C2(al)}
 C^{(2)}_{\al, \ga}:=\int_{0}^{1}\int_{0}^{\pi}\frac{r(1-r^{\ga})(1-r^{\al-2-\ga})}{\left(1-2r\cos\theta+r^2\right)^{\frac{2+\al}{2}}}(\sin\theta)^{\al-1}\,d\theta\,dr>0.
 \end{equation}
 Moreover, the value $\ga=\frac{\al-2}{2}$ maximises the constant \eqref{C2(al)} for fixed $\al$.
 \end{lem}

 \begin{proof}
 (i) Let $\va$ be the argument of $x$. We have
\begin{align*}
 &\int_{\R^2_+}(|x|^{\ga}-|y|^{\ga})\min\left\{x_2^{\al-1},y_{2}^{\al-1}\right\}k_{2, \alpha}^{(\varepsilon)}(x, y) \, dy\\
 &=|x|^{\ga-\al}\int_{\R^2_+}\frac{1-|y|^{\ga}}{|\hat{x}-y|^{2+\al}}\min\left\{x_2^{\al-1},|x|^{\al-1}y_{2}^{\al-1}\right\}\mathds{1}_{D_\varepsilon}(1, |y|)dy\\
 &=|x|^{\ga-\al}\left(\int_{0}^{1-\varepsilon}+\int_{\frac{1}{1-\varepsilon}}^{\infty}\right)\int_{0}^{\pi}\frac{r\left(1-r^{\ga}\right)}{\left(1-2r\cos(\va-\theta)+r^2\right)^{\frac{2+\al}{2}}}\min\left\{x_2^{\al-1},|x|^{\al-1}(r\sin\theta)^{\al-1}\right\}d\theta\,dr.
\end{align*}
Since $\al-1\ge 0$ and $\ga>0$, using $1\ge r^{\al-1}$ we can estimate the first (positive) term as follows,
\begin{align*}
&\int_{0}^{1-\varepsilon}\int_{0}^{\pi}\frac{r(1-r^{\ga})}{\left(1-2r\cos(\va-\theta)+r^2\right)^{\frac{2+\al}{2}}}\min\left\{x_2^{\al-1},|x|^{\al-1}(r\sin\theta)^{\al-1}\right\}d\theta\,dr\\
&\ge\int_{0}^{1-\varepsilon}\int_{0}^{\pi}\frac{r(1-r^{\ga})}{\left(1-2r\cos(\va-\theta)+r^2\right)^{\frac{2+\al}{2}}}\min\left\{r^{\al-1}x_2^{\al-1},|x|^{\al-1}(r\sin\theta)^{\al-1}\right\}d\theta\,dr\\
&=\int_{0}^{1-\varepsilon}\int_{0}^{\pi}\frac{r^{\al}(1-r^{\ga})}{\left(1-2r\cos(\va-\theta)+r^2\right)^{\frac{2+\al}{2}}}\min\left\{x_2^{\al-1},|x|^{\al-1}(\sin\theta)^{\al-1}\right\}d\theta\,dr.
\end{align*}
 In the second (negative) term, we substitute $r\mapsto\tfrac{1}{r}$ and use $r^{2(\al-1)}\le r^{\al-1}$ to obtain
 \begin{align*}
&\int_{\frac{1}{1-\varepsilon}}^{\infty}\int_{0}^{\pi}\frac{r\left(1-r^{\ga}\right)}{\left(1-2r\cos(\va-\theta)+r^2\right)^{\frac{2+\al}{2}}}\min\left\{x_2^{\al-1},|x|^{\al-1}(r\sin\theta)^{\al-1}\right\}d\theta\,dr\\
&=\int_{0}^{1-\varepsilon}\int_{0}^{\pi}\frac{r^{1-\al}(1-r^{-\ga})}{\left(1-2r\cos(\va-\theta)+r^2\right)^{\frac{2+\al}{2}}}\min\left\{r^{2(\al-1)}x_2^{\al-1},|x|^{\al-1}(r\sin\theta)^{\al-1}\right\}d\theta\,dr\\
&\ge\int_{0}^{1-\varepsilon}\int_{0}^{\pi}\frac{1-r^{-\ga}}{\left(1-2r\cos(\va-\theta)+r^2\right)^{\frac{2+\al}{2}}}\min\left\{x_2^{\al-1},|x|^{\al-1}(\sin\theta)^{\al-1}\right\}d\theta\,dr,
 \end{align*}
 Thus, adding the two estimates leads to 
 \begin{align*}
 &|x|^{\al-\ga}\int_{\R^2_+}(|x|^{\ga}-|y|^{\ga})\min\left\{x_2^{\al-1},y_{2}^{\al-1}\right\}k_{2, \alpha}^{(\varepsilon)}(x, y) \, dy\\
 &\ge\int_{0}^{1-\varepsilon}\int_{0}^{\pi}\frac{r^{\al}(1-r^{\ga})(1-r^{-\ga-\al})}{\left(1-2r\cos(\va-\theta)+r^2\right)^{\frac{2+\al}{2}}}\min\left\{x_2^{\al-1},|x|^{\al-1}(\sin\theta)^{\al-1}\right\}d\theta\,dr.
\end{align*}
Since $(1-r^{\ga})(1-r^{-\ga-\al})<0$ and $(\sin\theta)^{\al-1}\le 1$, we have by Lemma \ref{lem: bounds for H} applied to $\be=0$
\begin{align*}
&\int_{0}^{1-\varepsilon}\int_{0}^{\pi}\frac{r^{\al}(1-r^{\ga})(1-r^{-\ga-\al})}{\left(1-2r\cos(\va-\theta)+r^2\right)^{\frac{2+\al}{2}}}\min\left\{x_2^{\al-1},|x|^{\al-1}(\sin\theta)^{\al-1}\right\}d\theta\,dr\\
&\ge\int_{0}^{1-\varepsilon}\int_{0}^{\pi}\frac{r^{\al}(1-r^{\ga})(1-r^{-\ga-\al})}{\left(1-2r\cos(\va-\theta)+r^2\right)^{\frac{2+\al}{2}}}\min\left\{x_2^{\al-1},|x|^{\al-1}\right\}d\theta\,dr\\
&=x_2^{\al-1}\int_{0}^{1-\varepsilon}\int_{0}^{\pi}\frac{r^{\al}(1-r^{\ga})(1-r^{-\ga-\al})}{\left(1-2r\cos(\va-\theta)+r^2\right)^{\frac{2+\al}{2}}}d\theta\,dr\\
&\ge x_2^{\al-1}\int_{0}^{1-\varepsilon}\int_{0}^{\pi}\frac{r^{\al}(1-r^{\ga})(1-r^{-\ga-\al})}{\left(1-2r\sin\theta+r^2\right)^{\frac{2+\al}{2}}}d\theta\,dr.
\end{align*}
Notice that the condition $\al<2$ and Lemma \ref{integralonsphere} ensures the absolute convergence of the above integral for $\varepsilon\rightarrow 0^+$, since $|(1-r^{\ga})(1-r^{-\ga-\al})|$ is comparable with $(1-r)^2$, for $r\rightarrow 1^{-}$. Therefore, part (i) of the Lemma is proved.\\
\\
(ii) We proceed similarly as in the first part. Assuming $\ga<0$, in the first (negative) term 
$$
\int_{0}^{1-\varepsilon}\int_{0}^{\pi}\frac{r(1-r^{\ga})}{\left(1-2r\cos(\va-\theta)+r^2\right)^{\frac{2+\al}{2}}}\min\left\{x_2^{\al-1},|x|^{\al-1}(r\sin\theta)^{\al-1}\right\}d\theta\,dr
$$
we use the estimate $r^{\al-1}\le 1$, while in the second (positive) term
$$
\int_{0}^{1 - \varepsilon}\int_{0}^{\pi}\frac{r^{1-\al}(1-r^{-\ga})}{\left(1-2r\cos(\va-\theta)+r^2\right)^{\frac{2+\al}{2}}}\min\left\{r^{2(\al-1)}x_2^{\al-1},|x|^{\al-1}(r\sin\theta)^{\al-1}\right\}d\theta\,dr
$$
we use $r^{\al-1}\ge r^{2(\al-1)}$. Hence, adding these two terms, letting $\varepsilon \rightarrow 0^+$ and applying Lemma \ref{lem: bounds for H} for $\be=\al-1$, we obtain
\begin{align*}
  |x|^{\al-\ga}\liminf_{\ve\rightarrow 0^+}&\int_{\R^2_+}(|x|^{\ga}-|y|^{\ga})\min\left\{x_2^{\al-1},y_{2}^{\al-1}\right\}k_{2, \alpha}^{(\varepsilon)}(x, y) \, dy\\
&\ge\int_{0}^{1}\int_{0}^{\pi}\frac{r(1-r^{\ga})(1-r^{\al-2-\ga})}{\left(1-2r\cos(\va-\theta)+r^2\right)^{\frac{2+\al}{2}}}\min\left\{x_2^{\al-1},|x|^{\al-1}(\sin\theta)^{\al-1}\right\}\,d\theta\,dr\\
&\ge x_2^{\al-1}\int_{0}^{1}\int_{0}^{\pi}\frac{r(1-r^{\ga})(1-r^{\al-2-\ga})}{\left(1-2r\cos\theta+r^2\right)^{\frac{2+\al}{2}}}(\sin\theta)^{\al-1}\,d\theta\,dr,
 \end{align*}
 if $\al-2-\ga>0$. It is now elementary to check that the choice $\ga=\tfrac{\al-2}{2}$ maximises the constant above for fixed $\al$. That ends the proof.
 \end{proof}
Using the same argument as in the proof of Lemma \ref{lemma: weighted plaplacian}, we have the following corollary.
\color{black}
\begin{cor}\label{weighted laplacian Rd minimum}
 Let $1\le\al<2$, $x\in\R^d_+$,  $d\ge 2$, and $\om(x):=\left(x_{d-1}^2+x_d^2\right)^{\frac{\ga}{2}}=|\widetilde{x}|^{\ga}$. Then\\
 (i) for $0<\ga<1$,
 \begin{align*}
  \liminf_{\ve\rightarrow 0^+}&\int_{\R^d_+}(\om(x)-\om(y))\min\left\{x_d^{\al-1},y_{d}^{\al-1}\right\}k_{d, \al}^{(\varepsilon)}(x, y)dy
  \ge -\mathcal{A}_{d, \al}C^{(1)}_{\al, \ga}|\widetilde{x}|^{\ga-\al}x_d^{\al-1},
 \end{align*}
 where $C^{(1)}_{\al, \ga}$ is given by \eqref{C1(al,ga)};
 \\    
  (ii) if $\al-2<\ga<0$, then 
 \begin{align*}
  \liminf_{\ve\rightarrow 0^+}&\int_{\R^d_+}(\om(x)-\om(y))\min\left\{x_d^{\al-1},y_{d}^{\al-1}\right\}k_{d, \al}^{(\varepsilon)}(x, y)dy \ge\mathcal{A}_{d, \al}  C^{(2)}_{\al, \ga}|\widetilde{x}|^{\ga-\al}x_d^{\al-1},
 \end{align*}
 where $ C^{(2)}_{\al, \ga}$ is given by \eqref{C2(al)}.
\end{cor}

\begin{lem}\label{lem11}
 Let $0 < \al < 1$ and $x\in\R^2_+$. Then\\
 (i) for $0<\ga<\al$,
 \begin{align*}
  \liminf_{\ve\rightarrow 0^+}&\int_{\R^2_+}(|x|^{\ga}-|y|^{\ga})\min\left\{x_2^{\al-1},y_{2}^{\al-1}\right\}k_{2, \alpha}^{(\varepsilon)}(x, y) \, dy\ge -C^{(3)}_{\al,\ga}|x|^{\ga-1},
 \end{align*}
 where
 \begin{equation}\label{C3(al,ga)}
 C^{(3)}_{\al,\ga}:=\B\left(\tfrac{\alpha}{2}, \tfrac{1}{2}\right)\int_{0}^{1}\frac{r(1-r^{\ga})(r^{\al - \ga - 2}-1)}{(1-r)^{2+\al}}\, dr>0;
 \end{equation}
 \\
 (ii) if $-\alpha < \ga < 0$, then 
 \begin{align*}
  \liminf_{\ve\rightarrow 0^+}&\int_{\R^2_+}(|x|^{\ga}-|y|^{\ga})\min\left\{x_2^{\al-1},y_{2}^{\al-1}\right\}k_{2, \alpha}^{(\varepsilon)}(x, y) \, dy\ge C^{(4)}_{\al,\ga}|x|^{\ga-1},
 \end{align*}
 where
\begin{equation}\label{C4(al)}
 C^{(4)}_{\al,\ga} := \int_{0}^{1}\int_{0}^{\pi}\frac{\left(1-r^{-\ga}\right)\left(1-r^{\ga + \al} \right)}{\left(1-2r\cos\theta+r^2\right)^{\frac{2+\al}{2}}}d\theta\,dr>0.
 \end{equation}
 \end{lem}

Moreover, the choice $\ga=-\frac{\al}{2}$ maximises the constant \eqref{C4(al)} for fixed $\al$.

\begin{proof}
As in the proof of Lemma \ref{lem9}, we 
write 
\begin{align*}
    \int_{\R^2_+}(|x|^{\ga}-|y|^{\ga})\min\left\{x_2^{\al-1},y_{2}^{\al-1}\right\}k_{2, \alpha}^{(\varepsilon)}(x, y) \, dy = |x|^{\ga-\alpha}(I_1 + I_2)
\end{align*}
where $\va$ is the argument of $x$ and
\begin{align*}
    I_1 &:= \int_{0}^{1-\varepsilon}\int_{0}^{\pi}\frac{r(1-r^{\ga})}{\left(1-2r\cos(\va-\theta)+r^2\right)^{\frac{2+\al}{2}}}\min\left\{x_2^{\al-1},|x|^{\al-1}(r\sin\theta)^{\al-1}\right\}d\theta\,dr
    \\
    I_2 &:= \int_{0}^{1 - \varepsilon}\int_{0}^{\pi}\frac{r^{\al - 1}(1-r^{-\ga})}{\left(1-2r\cos(\va-\theta)+r^2\right)^{\frac{2+\al}{2}}}\min\left\{x_2^{\al-1},|x|^{\al-1}r^{1 - \al}(\sin\theta)^{\al-1}\right\}d\theta\,dr
\end{align*}
    (i) We have $r^{\alpha-1} \geq 1$ for $r \in (0, 1)$, hence
    $$I_1 \geq \int_{0}^{1-\varepsilon}\int_{0}^{\pi}\frac{r(1-r^{\ga})}{\left(1-2r\cos(\va-\theta)+r^2\right)^{\frac{2+\al}{2}}}\min\left\{x_2^{\al-1},|x|^{\al-1}(\sin\theta)^{\al-1}\right\}d\theta\,dr$$
and of course, as $r^{1-\alpha} \leq 1$,
\begin{align*}
    I_2 &\ge \int_{0}^{1-\varepsilon}\int_{0}^{\pi}\frac{r^{\al - 1}(1-r^{-\ga})}{\left(1-2r\cos(\va-\theta)+r^2\right)^{\frac{2+\al}{2}}}\min\left\{x_2^{\al-1},|x|^{\al-1}(\sin\theta)^{\al-1}\right\}d\theta\,dr
\end{align*}
 Thus, we get 
\begin{align*}
    I_1 + I_2 \ge &\int_{0}^{1-\varepsilon}\int_{0}^{\pi}\frac{r(1-r^{\ga})(1-r^{\al - \ga - 2})}{\left(1-2r\cos(\va-\theta)+r^2\right)^{\frac{2+\al}{2}}}\min\left\{x_2^{\al-1},|x|^{\al-1}(\sin\theta)^{\al-1}\right\}d\theta\,dr
    \\
    \ge &|x|^{\alpha-1}\int_{0}^{1-\varepsilon}\frac{r(1-r^{\ga})(1-r^{\al - \ga - 2})}{(1-r)^{2+\al}}\, dr \cdot \int_{0}^{\pi}(\sin\theta)^{\al-1}\, d\theta
    \\
    = &|x|^{\alpha-1}B\left(\tfrac{\alpha}{2}, \tfrac{1}{2}\right)\int_{0}^{1-\varepsilon}\frac{r(1-r^{\ga})(1-r^{\al - \ga - 2})}{(1-r)^{2+\al}}\, dr.
\end{align*}
Now taking the limit as $\varepsilon \rightarrow 0^+$ and using the monotone convergence theorem we prove (i).
To prove (ii), we get the lower bound in the analogous way:
$$I_1 \geq \int_{0}^{1-\varepsilon}\int_{0}^{\pi}\frac{r^{\al}(1-r^{\ga})}{\left(1-2r\cos(\va-\theta)+r^2\right)^{\frac{2+\al}{2}}}\min\left\{x_2^{\al-1},|x|^{\al-1}(\sin\theta)^{\al-1}\right\}d\theta\,dr$$
and
$$I_2 \geq \int_{0}^{1 - \varepsilon}\int_{0}^{\pi}\frac{1-r^{-\ga}}{\left(1-2r\cos(\va-\theta)+r^2\right)^{\frac{2+\al}{2}}}\min\left\{x_2^{\al-1},|x|^{\al-1}(\sin\theta)^{\al-1}\right\}d\theta\,dr.$$
By the fact that $(\sin\theta)^{\alpha-1} \ge 1$ for $\theta \in (0, \pi)$, by Lemma \ref{lem: bounds for H} for $\be=0$ we get
$$I_1 + I_2 \ge |x|^{\al-1}\int_{0}^{1-\varepsilon}\int_{0}^{\pi}\frac{(1-r^{-\ga})(1 - r^{\ga + \al})}{\left(1-2r\cos\theta+r^2\right)^{\frac{2+\al}{2}}}d\theta\,dr.$$
Again, by the monotone convergence theorem we conclude the proof of (ii).
\end{proof}
%
\begin{cor}\label{cor: liminf sp < 1}
 Let $0 < \al < 1$, $x\in\R^d_+$, $d\ge 2$, and $\om(x):=\left(x_{d-1}^2+x_d^2\right)^{\frac{\ga}{2}}=|\widetilde{x}|^{\ga}$. Then\\
 (i) for $0<\ga<\al$,
 \begin{align*}
  \liminf_{\ve\rightarrow 0^+}&\int_{\R^d_+}(\om(x)-\om(y))\min\left\{x_d^{\al-1},y_{d}^{\al-1}\right\}k_{d, \al}^{(\varepsilon)}(x, y)dy
  \ge -\mathcal{A}_{d, \al}C^{(3)}_{\al,\ga}|\widetilde{x}|^{\ga-1},
 \end{align*}
 where $C^{(3)}_{\al,\ga}$ is given by \eqref{C3(al,ga)};
 \\
 (ii) if $-\alpha < \ga < 0$, then 
 \begin{align*}
  \liminf_{\ve\rightarrow 0^+}&\int_{\R^d_+}(\om(x)-\om(y))\min\left\{x_d^{\al-1},y_{d}^{\al-1}\right\}k_{d, \al}^{(\varepsilon)}(x, y)dy
  \ge \mathcal{A}_{d, \al}C^{(4)}_{\al, \ga}|\widetilde{x}|^{\ga-1},
 \end{align*}
 where $C^{(4)}_{\al, \ga}$ is given by \eqref{C4(al)}.  
\end{cor}
\begin{rem}
Tracking the proofs of Lemmas \ref{lem9} and \ref{lem11}, we can easily see that they fit into the assumptions described in Remark \ref{rem3}. Hence, we can use these results in the ground-state representation. 
\end{rem}
\color{black}
The next lemma establishes a specific weighted fractional Hardy-type inequality, which will be needed in proving our main result, Theorem \ref{thm1}, for $1<p<2$.
\begin{lem}
Let $0<\al<2$, $d\ge 1$, $0<\ga<\max\{1,\al\}$, and $x\in\R^d_+$. Then, uniformly on compact sets contained in $\R^d_+$,
$$
2\lim_{\ve\rightarrow 0^+}\int_{\R^d_+}\frac{x_d^{-\ga}-y_d^{-\ga}}{|x-y|^{d+\al}}(x_d y_d)^{\ga}\min\left\{x_d^{\al-1},y_d^{\al-1}\right\}\,\mathds{1}_{D_\varepsilon}(x_d, y_d)dy=\bar{C}_{d,\al,\ga}x_d^{\ga-1},
$$
where
\begin{equation}\label{C(d,ga,al)}
\bar{C}_{d,\al,\ga}:=2\mathcal{A}_{d+1,\al-1}\int_{0}^{1}\frac{t^{\ga}\left(1-t^{-\ga}\right)^2\min\left\{1,t^{\al-1}\right\}}{(1-t)^{1+\al}}\,dt.
\end{equation}
Consequently, for $u\in C_c\left(\Rplus\right)$ we have a Hardy-type inequality
\begin{equation}\label{Hardy identity with min}
\int_{\Rplus}\int_{\Rplus} \frac{\left(u(x)-u(y)\right)^2}{|x-y|^{d+\al}}(x_d y_d)^{\ga}\min\left\{x_d^{\al-1},y_d^{\al-1}\right\}\,dy\,dx\ge \bar{C}_{d,\al,\ga}\int_{\Rplus}\frac{|u(x)|^2}{x_d^{1-2\ga}}dx.
\end{equation}
\end{lem}
\begin{proof}
Integrating over the first $d-1$ variables and using \eqref{integral identity}, we have 
\begin{align*}
&\int_{\R^d_+}\frac{x_d^{-\ga}-y_d^{-\ga}}{|x-y|^{d+\al}}(x_d y_d)^{\ga}\min\left\{x_d^{\al-1},y_d^{\al-1}\right\}\mathds{1}_{D_\varepsilon}(x_d, y_d)\,dy\\
&=\mathcal{A}_{d+1,\al-1}\int_{\R}\frac{x_d^{-\ga}-y_d^{-\ga}}{|x_d-y_d|^{1+\al}}(x_d y_d)^{\ga}\min\left\{x_d^{\al-1},y_d^{\al-1}\right\}\mathds{1}_{D_\varepsilon}(x_d, y_d)\,dy_d.
\end{align*}
The rest of the calculations is standard: we substitute $y_d=tx_d$, split the integral into two integrals on $(0,1-\varepsilon)$ and $((1 - \varepsilon)^{-1},\infty)$, substitute $t\mapsto\tfrac{1}{t}$ in the second and rearrange. The condition $0<\ga<\max\{1,\al\}$ ensures the convergence of both integrals, as well as the constant $\bar{C}_{d,\al,\ga}$ from \eqref{C(d,ga,al)}. We omit the details. The inequality \eqref{Hardy identity with min} follows from the ground-state method.
\end{proof}

\subsection{Proofs of the main results}
\begin{proof}[Proof of Theorem \ref{thm1}]
 The proof will be divided into two cases.\\
 \\
 \textbf{Case one: $p\ge 2$}. By the ground state representation \cite[Theorem 1.2]{MR2723817} applied to $\om(x)=x_d^{(sp-1)/p}$, we have
 \begin{align}\begin{split}\label{equation: gsr plaplacian}
\int_{\Rplus}\int_{\Rplus}\frac{|u(x)-u(y)|^p}{|x-y|^{d+sp}}dy\,dx&\ge\mathcal{D}_{d,s,p}\int_{\Rplus}\frac{|u(x)|^p}{x_d^{sp}}dx\\
&+c_p\int_{\Rplus}\int_{\Rplus}\frac{\left|v(x)-v(y)\right|^p}{|x-y|^{d+sp}}\left(x_d y_d\right)^{\frac{sp-1}{2}}dy\,dx,
 \end{split}
 \end{align}
 where $v(x)=u(x)x_d^{\frac{1-sp}{p}}$ and $c_p$ is given by \eqref{cp}. Again by the ground-state representation \eqref{General Hardy inequality} and Lemma \ref{lemma: weighted plaplacian} we have 
 \begin{align}\begin{split}\label{equation: gsr for plaplasian reminder}
     \int_{\Rplus}\int_{\Rplus}\frac{\left|v(x)-v(y)\right|^p}{|x-y|^{d+sp}}\left(x_d y_d\right)^{\frac{sp-1}{2}}dy\,dx \geq &\int_{\Rplus}|v(x)|^pV_{s, p, \ga}(x)dx
     \\
     + &c_p\int_{\Rplus}\int_{\Rplus}\frac{\left|\widetilde{v}(x)-\widetilde{v}(y)\right|^p}{|x-y|^{d+sp}}\left(x_d y_d\right)^{\frac{sp-1}{2}}|\widetilde{x}|^\frac{p\gamma}{2}|\widetilde{y}|^\frac{p\gamma}{2}dy\,dx,
\end{split}
 \end{align}
 for all $\gamma \in \left( -\frac{sp + 3}{2(p-1)}, \frac{sp+1}{2(p-1)} \right)$, where $\widetilde{v}(x) := v(x)|\widetilde{x}|^{-\ga}$,
 $$V_{s, p, \gamma}(x) := 2\mathcal{A}_{d, sp}|\widetilde{x}|^\frac{-sp-1}{2}x_d^{\frac{sp-1}{2}}F_{s, p, \beta_0, \gamma}(\hat{x}),$$
 $F_{s,p,\be,\ga}$ is given by \eqref{F(x)} and $\beta_0 := \frac{sp-1}{2}$. We observe that $F_{s, p, \beta_0, \gamma} > 0 \iff \gamma \in (-\frac{1}{p-1}, 0)$. Let $e_1=(1,0)$ and $e_2=(0,1)$. \color{black} We define
\begin{align}
    C^+_{s, p} := \begin{cases}
     &F_{s, p, \frac{sp-1}{2}, \gamma_0}(e_1), \quad sp \geq 1,
     \\
     &\B\left(\frac{sp+1}{4}, \frac{1}{2} \right)\displaystyle\int_0^1 \frac{r^\frac{sp+1}{2}|1-r^{\ga_0}|^{p-1} |1-r^{-1-\ga_0(p-1)}|}{(1+r)^{sp+2}}dr \\
     &=\B\left(\frac{sp+1}{4}, \frac{1}{2} \right)\displaystyle\int_0^1 \frac{r^\frac{sp+1}{2}\left|1-r^{-1/p}\right|^p}{(1+r)^{sp+2}}dr, \quad sp < 1,
 \end{cases}
\end{align}
where $\ga_0 := -\frac{1}{p}$ (this choice maximises $C^+_{s, p}$, see Remark \ref{rem8}). Also, for $\ga \in \left(0, \frac{sp+1}{2(p-1)}\right)$, we define
\begin{align}
    C^-_{s, p, \ga} := \begin{cases}
     \left|F_{s, p, \frac{sp-1}{2}, \gamma}(e_2)\right|, \quad &sp \geq 1, 
     \\
      \B\left(\frac{sp+1}{4}, \frac{1}{2} \right)\displaystyle\int_{0}^{1}\frac{r^\frac{sp+1}{2}|1-r^\ga|^{p-1} |1-r^{-1-\ga(p-1)}|}{(1-r)^{sp+2}}dr, \quad &sp < 1.
 \end{cases}
\end{align}
 By Lemma \ref{lem: bounds for H} and \eqref{bounds for H}, we have
 \begin{align*}
     V_{s, p, \gamma_0}(x) &\geq 2\mathcal{A}_{d, sp} C^+_{s, p}|\widetilde{x}|^\frac{-sp-1}{2}x_d^{\frac{sp-1}{2}},
     \\
     V_{s, p, \gamma}(x) &\geq -2\mathcal{A}_{d, sp}C^-_{s, p, \ga}|\widetilde{x}|^\frac{-sp-1}{2}x_d^{\frac{sp-1}{2}},
 \end{align*}
 for all $x \in \Rplus$. Combining \eqref{equation: gsr plaplacian} with \eqref{equation: gsr for plaplasian reminder} we get 
\begin{align*}
    \int_{\Rplus}\int_{\Rplus}\frac{|u(x)-u(y)|^p}{|x-y|^{d+sp}}dy\,dx&\ge\mathcal{D}_{d,s,p}\int_{\Rplus}\frac{|u(x)|^p}{x_d^{sp}}dx
    \\
    &+ \frac{c_p^2 C^+_{s, p}}{C^+_{s, p} + C^-_{s, p, \ga}}\\
    &\times\int_{\Rplus}\int_{\Rplus}\frac{\left|\widetilde{v}(x)-\widetilde{v}(y)\right|^p}{|x-y|^{d+sp}}\left(x_d y_d\right)^{\frac{sp-1}{2}}|\widetilde{x}|^\frac{p\ga}{2}|\widetilde{y}|^\frac{p\ga}{2}dy\,dx,
\end{align*}
By the weighted fractional Hardy inequality \eqref{weightedsharpfractionalHardyhalfspace}, for any $0<\ga<\min\left\{\tfrac{sp+1}{2(p-1)},\tfrac{sp+1}{p}\right\}=\tfrac{sp+1}{2(p-1)}$, we have
\begin{align*}
    \int_{\Rplus}\int_{\Rplus}&\frac{\left|\widetilde{v}(x)-\widetilde{v}(y)\right|^p}{|x-y|^{d+sp}}\left(x_d y_d\right)^{\frac{sp-1}{2}}|\widetilde{x}|^\frac{p\ga}{2}|\widetilde{y}|^\frac{p\ga}{2}dy\,dx 
     \\
     \geq
     &\int_{\Rplus}\int_{\Rplus}\frac{\left|\widetilde{v}(x)-\widetilde{v}(y)\right|^p}{|x-y|^{d+sp}}\left(x_d y_d\right)^{\frac{p(s + \ga)-1}{2}}dy\,dx 
     \\
     \geq &\mathcal{D}_{d, s, p, \frac{p(s + \ga)-1}{2}}\int_{\Rplus}\frac{|\widetilde{v}(x)|^p}{x_d^{1-p\ga}}dx = \mathcal{D}_{d, s, p, \frac{sp + \tau-1}{2}}\int_{\Rplus}\frac{|u(x)|^p}{x_d^{sp - \tau}\left(x_{d-1}^2 + x_d^2\right)^{\frac{\tau}{2}}}dx,
\end{align*}
where $\tau= p\ga$. Finally, we get
\begin{align*}
    \int_{\Rplus}\int_{\Rplus}\frac{|u(x)-u(y)|^p}{|x-y|^{d+sp}}dy\,dx&\ge\mathcal{D}_{d,s,p}\int_{\Rplus}\frac{|u(x)|^p}{x_d^{sp}}dx
    \\
    &+ C_{d, s, p, \tau}\int_{\Rplus}\frac{|u(x)|^p}{x_d^{sp - \tau}\left(x_{d-1}^2 + x_d^2\right)^{\frac{\tau}{2}}}dx,
\end{align*}
where
\begin{equation}\label{constant p=2}
C_{d, s, p, \tau} := \frac{c_p^2 \mathcal{D}_{d, s, p, \frac{sp + \tau-1}{2}} C^+_{s, p}}{C^+_{s, p} + C^-_{s, p, \frac{\tau}{p}}} > 0.
\end{equation}
\\
\\
 \textbf{Case two: $1<p<2$}. Since $|u(x)-u(y)|\ge||u(x)|-|u(y)||$, without loss of generality we may assume that $u$ is nonnegative. Then, by \eqref{General Hardy with weight1<p<2}, 
\begin{align*}
\int_{\Rplus}\int_{\Rplus}\frac{|u(x)-u(y)|^p}{|x-y|^{d+sp}}dy\,dx&\ge\mathcal{D}_{d,s,p}\int_{\Rplus}\frac{|u(x)|^p}{x_d^{sp}}dx\\
&+(p-1)\int_{\Rplus}\int_{\Rplus}\frac{\left(v(x)^{p/2}-v(y)^{p/2}\right)^2}{|x-y|^{d+sp}} W(x,y)\,dy\,dx,
\end{align*}
where, again, $v(x)=u(x)x_d^{\frac{1-sp}{p}}$ and 
$$
W(x,y)=\min\left\{x_d^{\frac{sp-1}{p}},y_d^{\frac{sp-1}{p}}\right\}\max\left\{x_d^{\frac{(p-1)(sp-1)}{p}},y_d^{\frac{(p-1)(sp-1)}{p}}\right\}.
$$
Observe that 
$$
W(x,y)\ge\min\left\{x_d^{sp-1},y_d^{sp-1}\right\}.
$$
We assume first that $sp\ge1$. By the ground-state representation and part (ii) of the Corollary \ref{weighted laplacian Rd minimum}, we have 
\begin{align}\label{ineq1}
\nonumber\int_{\Rplus}\int_{\Rplus}\frac{\left(v(x)^{p/2}-v(y)^{p/2}\right)^2}{|x-y|^{d+sp}} \min\left\{x_d^{sp-1},y_d^{sp-1}\right\}\,dy\,dx&\ge2\mathcal{A}_{d, sp}\widetilde{C}^+_{s, p}\int_{\Rplus}\frac{|v(x)|^p}{|\widetilde{x}|^{sp}x_d^{1-sp}}dx\\
&=2\mathcal{A}_{d, sp}\widetilde{C}^+_{s, p}\int_{\Rplus}\frac{|u(x)|^p}{|\widetilde{x}|^{sp}}dx,
\end{align}
where 
$$
\widetilde{C}^+_{s, p} :=  C^{(2)}_{sp,\frac{sp-2}{2}}
$$
and $C^{(2)}_{sp,\frac{sp-2}{2}}$ is given by \eqref{C2(al)}. Let now $0<\ga<1$. Applying part (i) of the Corollary \ref{weighted laplacian Rd minimum} results in the inequality
\begin{align}
\begin{split}\label{ineq2}
 &\int_{\Rplus}\int_{\Rplus}\frac{\left(v(x)^{p/2}-v(y)^{p/2}\right)^2}{|x-y|^{d+sp}} \min\left\{x_d^{sp-1},y_d^{sp-1}\right\}\,dy\,dx\\
 &\ge-2\mathcal{A}_{d, sp}\widetilde{C}^{-}_{s, p, \ga}\int_{\Rplus}\frac{|u(x)|^p}{|\widetilde{x}|^{sp}}\,dx+\int_{\Rplus}\int_{\Rplus}\frac{\left(\widetilde{v}(x)-\widetilde{v}(y)\right)^2}{|x-y|^{d+sp}}|\widetilde{x}|^{\ga}|\widetilde{y}|^{\ga}\min\left\{x_d^{sp-1},y_d^{sp-1}\right\}\,dy\,dx,
 \end{split}
\end{align}
with
$$
\widetilde{C}^{-}_{s, p, \ga}:=C^{(1)}_{sp, \ga},
$$
where $C^{(1)}_{sp, \ga}$ is given by \eqref{C1(al,ga)} and $\widetilde{v}(x):=v(x)^{p/2}|\widetilde{x}|^{-\ga}$. Therefore, combining \eqref{ineq1} and \eqref{ineq2}, we arrive at
\begin{align*}
 &\int_{\Rplus}\int_{\Rplus}\frac{\left(v(x)^{p/2}-v(y)^{p/2}\right)^2}{|x-y|^{d+sp}} \min\left\{x_d^{sp-1},y_d^{sp-1}\right\}\,dy\,dx\\
 &\ge\frac{\widetilde{C}^{+}_{s, p}}{\widetilde{C}^{+}_{s, p}+\widetilde{C}^{-}_{s, p, \ga}} \int_{\Rplus}\int_{\Rplus}\frac{\left(\widetilde{v}(x)-\widetilde{v}(y)\right)^2}{|x-y|^{d+sp}}|\widetilde{x}|^{\ga}|\widetilde{y}|^{\ga}\min\left\{x_d^{sp-1},y_d^{sp-1}\right\}\,dy\,dx,
\end{align*}
Using \eqref{Hardy identity with min}, we have 
\begin{align*}
\int_{\Rplus}\int_{\Rplus}&\frac{\left(\widetilde{v}(x)-\widetilde{v}(y)\right)^2}{|x-y|^{d+sp}}|\widetilde{x}|^{\ga}|\widetilde{y}|^{\ga}\min\left\{x_d^{sp-1},y_d^{sp-1}\right\}\,dy\,dx\\
&\geq\int_{\Rplus}\int_{\Rplus}\frac{\left(\widetilde{v}(x)-\widetilde{v}(y)\right)^2}{|x-y|^{d+sp}}\left(x_d y_d\right)^{\ga}\min\left\{x_d^{sp-1},y_d^{sp-1}\right\}\,dy\,dx\\
&\ge \bar{C}_{d,sp,\ga}\int_{\Rplus}\frac{|\widetilde{v}(x)|^2}{x_d^{1-2\ga}}dx\\
&=\bar{C}_{d,sp,\ga}\int_{\Rplus}\frac{|u(x)|^p}{x_d^{sp-2\ga}|\widetilde{x}|^{2\ga}}dx,
\end{align*}
where $\bar{C}_{d,sp,\ga}$ is given by \eqref{C(d,ga,al)} with $\al=sp$. We now write $\tau=2\ga\in (0,2)$. Hence, we proved \eqref{mainresult1} with 
\begin{align}\label{equation: C for plaplacian case one}
    C_{d,s,p,\tau} :=\frac{(p-1)\widetilde{C}^{+}_{s, p}\bar{C}_{d, sp, \frac{\tau}{2}}}{\widetilde{C}^{+}_{s, p}+\widetilde{C}^{-}_{s, p, \frac{\tau}{2}}}>0.
\end{align}
Let us now assume that $sp<1$. Let $\ga_0 := -\frac{sp}{2}$. By Corollary \ref{cor: liminf sp < 1} we have 
\begin{align}
\begin{split}\label{ineq1 sp <1}
    \int_{\Rplus}\int_{\Rplus}\frac{\left(v(x)^{p/2}-v(y)^{p/2}\right)^2}{|x-y|^{d+sp}} \min\left\{x_d^{sp-1},y_d^{sp-1}\right\}\,dy\,dx&\ge2\mathcal{A}_{d, sp}\widehat{C}^+_{s, p}\int_{\Rplus}\frac{|v(x)|^p}{|\widetilde{x}|}dx,
\end{split}
\end{align}
where 
$$\widehat{C}^+_{s, p} := C^{(4)}_{sp, -\frac{sp}{2}}$$
and $C^{(4)}_{sp, -\frac{sp}{2}}$ is given by \eqref{C4(al)}. Also by Corollary \ref{cor: liminf sp < 1}, for $\ga \in (0, sp)$ we have 
\begin{align}
\begin{split}\label{ineq2 sp <1}
 &\int_{\Rplus}\int_{\Rplus}\frac{\left(v(x)^{p/2}-v(y)^{p/2}\right)^2}{|x-y|^{d+sp}} \min\left\{x_d^{sp-1},y_d^{sp-1}\right\}\,dy\,dx\\
 &\ge-2\mathcal{A}_{d, sp}\widehat{C}^{-}_{s, p, \ga}\int_{\Rplus}\frac{|v(x)|^p}{|\widetilde{x}|}\,dx+\int_{\Rplus}\int_{\Rplus}\frac{\left(\widetilde{v}(x)-\widetilde{v}(y)\right)^2}{|x-y|^{d+sp}}|\widetilde{x}|^{\ga}|\widetilde{y}|^{\ga}\min\left\{x_d^{sp-1},y_d^{sp-1}\right\}\,dy\,dx,
 \end{split}
\end{align}
with
$$\widehat{C}^-_{s, p, \ga} := C^{(3)}_{sp,\ga},$$
where $C^{(3)}_{sp,\ga}$ is given by \eqref{C3(al,ga)} and $\widetilde{v}(x) := v(x)^{p/2}|\widetilde{x}|^{-\ga}$. Similarly as in previous cases, combining \eqref{ineq1 sp <1} with \eqref{ineq2 sp <1} and applying \eqref{Hardy identity with min}, we get
\begin{align*}
 &\int_{\Rplus}\int_{\Rplus}\frac{\left(v(x)^{p/2}-v(y)^{p/2}\right)^2}{|x-y|^{d+sp}} \min\left\{x_d^{sp-1},y_d^{sp-1}\right\}\,dy\,dx\\
 &\ge\frac{\widehat{C}^{+}_{s, p}}{\widehat{C}^{+}_{s, p}+\widehat{C}^{-}_{s, p, \ga}} \int_{\Rplus}\int_{\Rplus}\frac{\left(\widetilde{v}(x)-\widetilde{v}(y)\right)^2}{|x-y|^{d+sp}}|\widetilde{x}|^{\ga}|\widetilde{y}|^{\ga}\min\left\{x_d^{sp-1},y_d^{sp-1}\right\}\,dy\,dx
 \\
 &\geq \frac{\widehat{C}^{+}_{s, p}\bar{C}_{d, sp, \ga}}{\widehat{C}^{+}_{s, p}+\widehat{C}^{-}_{s, p, \ga}}\int_{\Rplus}\frac{|\widetilde{v}(x)|^2}{x_d^{1- 2\ga}}dx
 \\
 &= \frac{\widehat{C}^{+}_{s, p}\bar{C}_{d, sp, \tau/2}}{\widehat{C}^{+}_{s, p}+\widehat{C}^{-}_{s, p, \tau/2}}\int_{\Rplus}\frac{|u(x)|^p}{x_d^{sp- \tau}\left( x_{d-1}^2 + x_d^2\right)^{\frac{\tau}{2}}}dx,
\end{align*}
where $\tau = 2\ga$. Collecting all the above inequalities gives the desired result with the constant 
$$C_{d,s,p,\tau} := \frac{(p-1)\widehat{C}^{+}_{s, p}\bar{C}_{d, sp, \frac{\tau}{2}}}{\widehat{C}^{+}_{s, p}+\widehat{C}^{-}_{s, p, \frac{\tau}{2}}}.$$
The proof is now complete.
\end{proof}
\begin{proof}[Proof of Theorem \ref{thm2}] By \eqref{groundstateSobolevBregman}, we have the following Hardy identity for Sobolev--Bregman form:
\begin{align}\label{critical hardy for SB}
E_p[u]=\mathcal{D}'_{d,\alpha,p}\int_{\Rplus}\frac{|u(x)|^p}{x_d^{\alpha}}\,dx+\frac{2}{p}\int_{\Rplus}\int_{\Rplus}\frac{F_p\left(v(x),v(y)\right)x_d^{\frac{\alpha-1}{q}}y_d^{\frac{\alpha-1}{p}}}{|x - y|^{\alpha+d}}\,dx\,dy,
\end{align}
where $q=\tfrac{p}{p-1}$ and $v(x) := u(x)x_d^{\frac{1-\alpha}{p}}$. Let 
\begin{equation}\label{cp*}
    c_p^* := \inf_{t \in \R\setminus\{1\}} \frac{|t|^p-1-p(t-1)}{\left(t^{\langle\frac{p}{2}\rangle} - 1\right)^2} > 0.
\end{equation}    
It then holds $F_p(a,b)\ge c_p^{*}\left(a^{\langle p/2\rangle}-b^{\langle p/2\rangle}\right)^2$, $a,b\in\R$. Hence, we observe that
\begin{align*}
    &\int_{\Rplus}\int_{\Rplus}\frac{F_p\left(v(x),v(y)\right)x_d^{\frac{\alpha-1}{q}}y_d^{\frac{\alpha-1}{p}}}{|x - y|^{\alpha+d}}\,dx\,dy
    \\
    \geq&c_p^*\int_{\Rplus}\int_{\Rplus}\frac{\left( v(x)^{\langle\frac{p}{2} \rangle} - v(y)^{\langle \frac{p}{2} \rangle} \right)^2x_d^{\frac{\alpha-1}{q}}y_d^{\frac{\alpha-1}{p}}}{|x - y|^{\alpha+d}}\,dx\,dy
    \\
    = &\frac{c_p^*}{2}\int_{\Rplus}\int_{\Rplus}\frac{\left( v(x)^{\langle\frac{p}{2} \rangle} - v(y)^{\langle \frac{p}{2} \rangle} \right)^2\Big(x_d^{\frac{\alpha-1}{q}}y_d^{\frac{\alpha-1}{p}} + x_d^{\frac{\alpha-1}{p}}y_d^{\frac{\alpha-1}{q}}\Big)}{|x - y|^{\alpha+d}}\,dx\,dy
    \\
    \geq&c_p^*\int_{\Rplus}\int_{\Rplus}\frac{\left( v(x)^{\langle\frac{p}{2} \rangle} - v(y)^{\langle \frac{p}{2} \rangle} \right)^2 }{|x - y|^{\alpha+d}}(x_dy_d)^{\frac{\alpha-1}{2}}\,dx\,dy,
\end{align*}
where in the last passage we applied the elementary inequality $a^2+b^2\ge 2ab$. Thus, we get 
\begin{align}
 \begin{split}\label{equation: reduced GSR SB}
E_p[u]\ge&\mathcal{D}'_{d,\alpha,p}\int_{\Rplus}\frac{|u(x)|^p}{x_d^{\alpha}}\,dx
    \\
    +&\frac{2c_p^*}{p}\int_{\Rplus}\int_{\Rplus}\frac{\left( v(x)^{\langle\frac{p}{2} \rangle} - v(y)^{\langle \frac{p}{2} \rangle} \right)^2 }{|x - y|^{\alpha+d}}(x_dy_d)^{\frac{\alpha-1}{2}}\,dx\,dy.
 \end{split} 
\end{align} Observe that the expression
$$\int_{\Rplus}\int_{\Rplus}\frac{\left( v(x)^{\langle\frac{p}{2} \rangle} - v(y)^{\langle \frac{p}{2} \rangle} \right)^2 }{|x - y|^{\alpha+d}}(x_dy_d)^{\frac{\alpha-1}{2}}\,dx\,dy.$$
coincides with the remainder term in \eqref{equation: gsr plaplacian} in the quadratic case, with $v$ replaced by $v^{\langle\frac{p}{2}\rangle}$. Now, following the same steps as in case one of the proof of Theorem $\ref{thm1}$ we arrive at the desired result with the constant
$$C'_{d, \alpha, p, \tau} := \frac{2c_p^*}{p}C_{d, \frac{\al}{2}, 2, \tau},$$
where $C_{d, \frac{\al}{2}, 2, \tau}$ is given by \eqref{constant p=2}.
\end{proof}


\begin{thebibliography}{10}

\bibitem{MR5026388}
{\sc Adimurthi, Jana, P., and Roy, P.}
\newblock Boundary fractional {H}ardy's inequality in dimension one: the critical case.
\newblock {\em Commun. Contemp. Math. 28}, 4 (2026), Paper No. 2550051, 14.

\bibitem{MR3807591}
{\sc Adimurthi, and Mallick, A.}
\newblock A {H}ardy type inequality on fractional order {S}obolev spaces on the {H}eisenberg group.
\newblock {\em Ann. Sc. Norm. Super. Pisa Cl. Sci. (5) 18}, 3 (2018), 917--949.

\bibitem{MR5018356}
{\sc Adimurthi, Roy, P., and Sahu, V.}
\newblock Fractional boundary {H}ardy inequality for the critical cases.
\newblock {\em J. Funct. Anal. 290}, 8 (2026), Paper No. 111351, 63.

\bibitem{MR5015195}
{\sc Adimurthi, Roy, P., and Sahu, V.}
\newblock Fractional {H}ardy inequality with singularity on submanifold.
\newblock {\em Calc. Var. Partial Differential Equations 65}, 2 (2026), Paper No. 62, 40.

\bibitem{MR4815911}
{\sc Anoop, T.~V., Roy, P., and Roy, S.}
\newblock On fractional {O}rlicz-{H}ardy inequalities.
\newblock {\em J. Math. Anal. Appl. 543}, 2 (2025), Paper No. 128980, 29.

\bibitem{MR1254832}
{\sc Beckner, W.}
\newblock Pitt's inequality and the uncertainty principle.
\newblock {\em Proc. Amer. Math. Soc. 123}, 6 (1995), 1897--1905.

\bibitem{MR2424899}
{\sc Benguria, R.~D., Frank, R.~L., and Loss, M.}
\newblock The sharp constant in the {H}ardy-{S}obolev-{M}az'ya inequality in the three dimensional upper half-space.
\newblock {\em Math. Res. Lett. 15}, 4 (2008), 613--622.

\bibitem{MR4800921}
{\sc Bianchi, F., Brasco, L., and Zagati, A.~C.}
\newblock On the sharp {H}ardy inequality in {S}obolev-{S}lobodecki\u i\ spaces.
\newblock {\em Math. Ann. 390}, 1 (2024), 493--555.

\bibitem{MR2663757}
{\sc Bogdan, K., and Dyda, B.}
\newblock The best constant in a fractional {H}ardy inequality.
\newblock {\em Math. Nachr. 284}, 5-6 (2011), 629--638.

\bibitem{MR4589708}
{\sc Bogdan, K., Grzywny, T., Pietruska-Pałuba, K., and Rutkowski, A.}
\newblock Nonlinear nonlocal {D}ouglas identity.
\newblock {\em Calc. Var. Partial Differential Equations 62}, 5 (2023), Paper No. 151, 31.

\bibitem{MR4851904}
{\sc Bogdan, K., Gutowski, M., and Pietruska-Pałuba, K.}
\newblock Polarized {H}ardy-{S}tein identity.
\newblock {\em J. Funct. Anal. 288}, 7 (2025), Paper No. 110827, 39.

\bibitem{MR4372148}
{\sc Bogdan, K., Jakubowski, T., Lenczewska, J., and Pietruska-Pałuba, K.}
\newblock Optimal {H}ardy inequality for the fractional {L}aplacian on {$L^p$}.
\newblock {\em J. Funct. Anal. 282}, 8 (2022), Paper No. 109395, 31.

\bibitem{DiebTemgoua2026arxiv}
{\sc Dieb, A., and Temgoua, R.~Y.}
\newblock Fractional {H}ardy inequalities on {$C^{1,1}$} open sets.
\newblock arXiv preprint arXiv:2602.10463, 2026.

\bibitem{MR2755892}
{\sc Dyda, B.}
\newblock Fractional {H}ardy inequality with a remainder term.
\newblock {\em Colloq. Math. 122}, 1 (2011), 59--67.

\bibitem{MR2910984}
{\sc Dyda, B., and Frank, R.~L.}
\newblock Fractional {H}ardy-{S}obolev-{M}az'ya inequality for domains.
\newblock {\em Studia Math. 208}, 2 (2012), 151--166.

\bibitem{MR4454384}
{\sc Dyda, B., and Kijaczko, M.}
\newblock On density of compactly supported smooth functions in fractional {S}obolev spaces.
\newblock {\em Ann. Mat. Pura Appl. (4) 201}, 4 (2022), 1855--1867.

\bibitem{MR4708667}
{\sc Dyda, B., and Kijaczko, M.}
\newblock Sharp fractional {H}ardy inequalities with a remainder for {$1 < p < 2$}.
\newblock {\em J. Funct. Anal. 286}, 9 (2024), Paper No. 110373, 19.

\bibitem{MR4705882}
{\sc Dyda, B., and Kijaczko, M.}
\newblock Sharp weighted fractional {H}ardy inequalities.
\newblock {\em Studia Math. 274}, 2 (2024), 153--171.

\bibitem{MR3803664}
{\sc Dyda, B., Lehrbäck, J., and Vähäkangas, A.~V.}
\newblock Fractional {H}ardy-{S}obolev type inequalities for half spaces and {J}ohn domains.
\newblock {\em Proc. Amer. Math. Soc. 146}, 8 (2018), 3393--3402.

\bibitem{MR3237044}
{\sc Dyda, B., and V\"ah\"akangas, A.~V.}
\newblock A framework for fractional {H}ardy inequalities.
\newblock {\em Ann. Acad. Sci. Fenn. Math. 39}, 2 (2014), 675--689.

\bibitem{MR4597627}
{\sc Fischer, F.}
\newblock A non-local quasi-linear ground state representation and criticality theory.
\newblock {\em Calc. Var. Partial Differential Equations 62}, 5 (2023), Paper No. 163, 33.

\bibitem{MR2863763}
{\sc Frank, R.~L., and Loss, M.}
\newblock Hardy-{S}obolev-{M}az'ya inequalities for arbitrary domains.
\newblock {\em J. Math. Pures Appl. (9) 97}, 1 (2012), 39--54.

\bibitem{MR2469027}
{\sc Frank, R.~L., and Seiringer, R.}
\newblock Non-linear ground state representations and sharp {H}ardy inequalities.
\newblock {\em J. Funct. Anal. 255}, 12 (2008), 3407--3430.

\bibitem{MR2723817}
{\sc Frank, R.~L., and Seiringer, R.}
\newblock Sharp fractional {H}ardy inequalities in half-spaces.
\newblock In {\em Around the research of {V}ladimir {M}az'ya. {I}}, vol.~11 of {\em Int. Math. Ser. (N. Y.)}. Springer, New York, 2010, pp.~161--167.

\bibitem{MR2360010}
{\sc Gradshteyn, I.~S., and Ryzhik, I.~M.}
\newblock {\em Table of integrals, series, and products}, seventh~ed.
\newblock Elsevier/Academic Press, Amsterdam, 2007.

\bibitem{MR4885983}
{\sc Gutowski, M., and Kwa\'snicki, M.}
\newblock Beurling-{D}eny formula for {S}obolev-{B}regman forms.
\newblock {\em Nonlinear Anal. 257\/} (2025), Paper No. 113808, 14.

\bibitem{MR2442182}
{\sc Han, J., Niu, P., and Qin, W.}
\newblock Hardy inequalities in half spaces of the {H}eisenberg group.
\newblock {\em Bull. Korean Math. Soc. 45}, 3 (2008), 405--417.

\bibitem{MR436854}
{\sc Herbst, I.~W.}
\newblock Spectral theory of the operator {$(p\sp{2}+m\sp{2})\sp{1/2}-Ze\sp{2}/r$}.
\newblock {\em Comm. Math. Phys. 53}, 3 (1977), 285--294.

\bibitem{Kijaczko2025arxiv}
{\sc Kijaczko, M.}
\newblock Best constants for {H}ardy inequalities in {T}riebel--{L}izorkin spaces.
\newblock arXiv preprint arXiv:2512.18688, 2025.

\bibitem{MR4720167}
{\sc Kijaczko, M., and Lenczewska, J.}
\newblock Sharp {H}ardy inequalities for {S}obolev-{B}regman forms.
\newblock {\em Math. Nachr. 297}, 2 (2024), 549--559.

\bibitem{KijaczkoSahu2026}
{\sc Kijaczko, M., and Sahu, V.}
\newblock Weighted fractional {H}ardy--{S}obolev and {H}ardy--{S}obolev--{M}az'ya inequalities with singularities on flat submanifold.
\newblock {\em Commun. Contemp. Math.\/} (2026).
\newblock Paper No. 2650016.

\bibitem{MR2659764}
{\sc Loss, M., and Sloane, C.}
\newblock Hardy inequalities for fractional integrals on general domains.
\newblock {\em J. Funct. Anal. 259}, 6 (2010), 1369--1379.

\bibitem{Mazya1972}
{\sc Maz'ya, V.}
\newblock On a degenerating problem with directional derivative.
\newblock {\em Mat. Sb. (N.S.) 87(129)\/} (1972), 417--454.
\newblock English translation: {M}ath. {USSR} {S}b. \textbf{16} (1972), no.~3, 429--469.

\bibitem{MR2755141}
{\sc Maz'ya, V.}
\newblock {\em Sobolev spaces: with applications to elliptic partial differential equations}, second, augmented~ed., vol.~46 of {\em Ergebnisse der Mathematik und ihrer Grenzgebiete. 3. Folge}.
\newblock Springer, Heidelberg, 2011.

\bibitem{MR2508844}
{\sc Maz'ya, V., and Shaposhnikova, T.}
\newblock A collection of sharp dilation invariant integral inequalities for differentiable functions.
\newblock In {\em Sobolev spaces in mathematics. {I}}, vol.~8 of {\em Int. Math. Ser. (N. Y.)}. Springer, New York, 2009, pp.~223--247.

\bibitem{MR2400106}
{\sc Pinchover, Y., Tertikas, A., and Tintarev, K.}
\newblock A {L}iouville-type theorem for the {$p$}-{L}aplacian with potential term.
\newblock {\em Ann. Inst. H. Poincar\'e{} C Anal. Non Lin\'eaire 25}, 2 (2008), 357--368.

\bibitem{MR3736851}
{\sc Psaradakis, G.}
\newblock Hardy-{S}obolev-{M}az'ya and related inequalities in the half-space.
\newblock {\em J. Elliptic Parabol. Equ. 3}, 1-2 (2017), 127--134.

\bibitem{MR4908058}
{\sc Rawat, R., and Roy, H.}
\newblock Fractional {H}ardy's inequality for half-spaces in the {H}eisenberg group.
\newblock {\em J. Math. Anal. Appl. 551}, 1 (2025), Paper No. 129674, 15.

\bibitem{MR4849884}
{\sc Sahu, V.}
\newblock Weighted fractional {H}ardy inequalities with singularity on any flat submanifold.
\newblock {\em J. Math. Anal. Appl. 546}, 2 (2025), Paper No. 129227, 16.

\bibitem{MR2823046}
{\sc Sloane, C.~A.}
\newblock A fractional {H}ardy-{S}obolev-{M}az'ya inequality on the upper halfspace.
\newblock {\em Proc. Amer. Math. Soc. 139}, 11 (2011), 4003--4016.

\bibitem{MR2124873}
{\sc Tidblom, J.}
\newblock A {H}ardy inequality in the half-space.
\newblock {\em J. Funct. Anal. 221}, 2 (2005), 482--495.

\bibitem{MR1717839}
{\sc Yafaev, D.}
\newblock Sharp constants in the {H}ardy-{R}ellich inequalities.
\newblock {\em J. Funct. Anal. 168}, 1 (1999), 121--144.

\end{thebibliography}
\end{document}